\documentclass[11pt,reqno]{amsart}
\usepackage{amsmath}
\usepackage{latexsym}
\usepackage{amsfonts}
 \numberwithin{equation}{section}
\begin{document}
\renewcommand{\theequation}
{\thesection.\arabic{equation}}
\newtheorem{Pa}{Paper}[section]
\newtheorem{theorem}[Pa]{{\bf Theorem}}
\newtheorem{lemma}[Pa]{{\bf Lemma}}
\newtheorem{definition}[Pa]{{\bf Definition}}
\newtheorem{corollary}[Pa]{{\bf Corollary}}
\newtheorem{Ob}[Pa]{{\bf Observation}}
\newtheorem{remark}[Pa]{{\bf Remark}}
\newtheorem{proposition}[Pa]{{\bf Proposition}}
\newtheorem{problem}[Pa]{{\bf Problem}}
\newtheorem{Ee}[Pa]{{\bf Exercise}}
\newtheorem{example}[Pa]{{\bf Example}}
\newtheorem{Qn}[Pa]{{\bf Question}}
\newtheorem{Ft}[Pa]{{\bf Fact}}
\newtheorem{Hyp}[Pa]{{\bf Hypothesis}}

\newcommand{\f}{w}
\newcommand{\num}{S}
\newcommand{\den}{B}
\newcommand{\hE}{G}
\newcommand{\s}{{\mathbb S}}
\newcommand{\bgam}{\boldsymbol{\gamma}}
\newcommand{\hC}{Y}  
\newcommand{\hTheta}{\widehat{\Theta}}
\newcommand{\cS}{{\mathcal {S}}}
\newcommand{\cR}{{\mathcal {RS}}}
\newcommand{\cB}{{\mathcal {B}}}
\newcommand{\cE}{{\mathcal {E}}}
\newcommand{\PR}{{\bf IP}_{\kappa}}
\renewcommand{\theequation}{\thesection.\arabic{equation}}
\newcommand{\C}{{\mathbb C}}
\newcommand{\D}{{\mathbb D}}
\newcommand{\T}{{\mathbb T}}
\newcommand{\N}{{\mathbb N}}
\newcommand{\PP}{\mathbb P}
\newcommand{\bk}{{\mathbf k}}
\newcommand{\bl}{{\boldsymbol\ell}}
\newcommand{\tC}{\widetilde{C}}
\newcommand{\tP}{\widetilde{P}}
\newcommand{\tT}{\widetilde{T}}
\newcommand{\tE}{\widetilde{E}}
\newcommand{\bn}{{\mathbf n}}
\newcommand{\bd}{{\mathbf d}}
\newcommand{\sbm}[1]{\left[\begin{smallmatrix} #1
                \end{smallmatrix}\right]}

\title[Boundary interpolation]
{Boundary interpolation by finite Blaschke products} 
\author{Vladimir Bolotnikov}
\address{Department of Mathematics\\
The College of William and Mary \\
Williamsburg, VA 23187-8795\\
USA}
\email{vladi@math.wm.edu}


\begin{abstract}
Given $n$ distinct points $t_1,\ldots,t_n$ on the unit circle $\T$ and equally many target values 
$\f_1,\ldots,\f_n\in\T$, we describe all Blaschke products $f$ of degree at most $n-1$ 
such that 
$f(t_i)=\f_i$ for $i=1,\ldots,n$. We also describe the cases where degree $n-1$ is the minimal possible.
\end{abstract}
\maketitle

\section{Introduction}
\setcounter{equation}{0} 

Let $\mathcal S$ denote the Schur class of all analytic functions mapping the open unit 
disk $\D=\{z\in\C: \, |z|<1\}$ 
into the closed unit disk $\overline{\D}$. We denote by $\mathcal{RS}_k$ the set of 
all rational $\cS$-class functions of degree at most $k$. The functions $f\in\mathcal{RS}_k$ that are 
unimodular on the unit circle $\T$ are necessarily of the form 
$$
f(z)=c\cdot \prod_{i=1}^k\frac{z-a_i}{1-z\overline{a}_i},\qquad |c|=1, \; |a_i|<1
$$
and are called {\em finite Blaschke products}. They can be characterized as Schur-class functions
that extend to a mapping from $\overline{\D}$ onto itself and then the degree $k=\deg f$ can be 
interpreted geometrically as the winding number of the image $f(\T)$ of the unit circle $\T$ about the 
origin (i.e., the map $f:\, \overline{\D}\to\overline{\D}$ is $k$-to-$1$). We will write $\cB_k$
for the set of all Blaschke products of degree at most $k$, and we will use 
notation $\cB^\circ_k:=\cB_k\backslash\cB_{k-1}$ for the set of Blaschke products of degree $k$.

\smallskip

Given points $z_1,\ldots,z_n\in\D$ and target values $w_1,\ldots,w_n\in\overline{\D}$, the classical 
{\em Nevanlinna-Pick problem} consists of finding a function $f\in\cS$ such that
\begin{equation}
f(z_i)=w_i\quad\mbox{for}\quad i=1,\ldots,n\quad (z_i\in\D, \; w_i\in\overline{\D}).
\label{1.1}
\end{equation}
The problem has a solution if and only if the matrix 
$P=\left[\frac{1-w_i\overline{w}_j}{1-z_i\overline{z}_j}\right]_{i,j=1}^n$ is 
positive semidefinite ($P\ge 0$) \cite{pick, nevan1}. If $\det P=0$, the problem
has a unique solution which is a finite Blaschke product of degree $k={\rm rank} \, P$. If $P$ is 
positive definite ($P>0$), the problem is {\em indeterminate} (has infinitely many solutions), and its 
solution set admits a linear-fractional parametrization with free Schur-class parameter.
When the parameter runs through the class $\cB^\circ_\kappa$, the parametrization formula produces all 
Blaschke-product solutions to the problem \eqref{1.1} of degree $n+\kappa$ for each fixed $\kappa\ge 0$.

\smallskip

If the problem \eqref{1.1} is indeterminate, it has no solutions in $\cB_{n-1}$. However,
it still has solutions in $\mathcal{RS}_{n-1}$ ({\em low-degree solutions}).
In case not all target values $w_i$'s are the same, the problem has infinitely many low-degree 
solutions which can be parametrized by
polynomials $\sigma$ with $\deg \sigma< n$ and with all the roots outside $\D$; see 
\cite{blgm}, \cite{geor}, \cite{geor1}.
More precisely, for every such $\sigma$, there exists a  unique (up to a common
unimodular  constant factor) pair of polynomials $a(z)$ and $b(z)$, each of degree
at most $n-1$  and such that
\begin{enumerate}
\item $|a(z)|^2-|b(z)|^2=|\sigma(z)|^2$ for $|z|=1$ and
\item the function $f=b/a$ (which belongs to $\mathcal{RS}_{n-1}$ by part (1))
satisfies \eqref{1.1}.
\end{enumerate}
The question of finding a rational solution of the minimal possible degree  $k_{\rm min}$ (and even 
finding the value of  $k_{\rm min}$) is still open.

\smallskip

The boundary version of the Nevanlinna-Pick problem interpolates preassigned values $\f_1,\ldots,\f_n$
(interpreted as nontangential boundary limits in the non-rational case) at finitely many 
points $t_1,\ldots,t_n$ on the unit circle $\T$. Obvious necessary conditions $|\f_i|\le 1$
($1\le i\le n$) turn out to be sufficient, and a solvable problem is always indeterminate. 
If at least one of the preassigned boundary values $\f_i$ is not unimodular, the problem cannot 
be solved by a finite Blaschke product. However, as in the classical ``interior" case, the problem admits
infinitely many low-degree rational solutions $f\in\mathcal{RS}_{n-1}$ (unless all target values $\f_i$'s 
are equal to each other); see e.g., \cite{bc}. As for now, the description of all low-degree solutions 
is not known.

\smallskip

In this paper we will focus on the ``boundary-to-boundary" Nevanlinna-Pick problem where all 
preassigned boundary values are unimodular. This problem can be solved by a finite Blaschke product.
\begin{theorem}
Given any points $t_1,\ldots,t_n\in\T$ and  $\f_1,\ldots,\f_n\in\T$, there exists a finite Blaschke
product $f\in\cB_{n-1}$ such that
\begin{equation}
f(t_i)=\f_i\quad\mbox{for}\quad i=1,\ldots,n \quad (t_i, \, \f_i\in\T).
\label{1.2}
\end{equation}
Any rational function $f\in\mathcal{RS}_{n-1}$ satisfying conditions \eqref{1.2} is necessarily a 
Blaschke product (i.e., $f\in \cB_{n-1}$).
\label{T:1.1}
\end{theorem}
The existence of a finite Blaschke product satisfying conditions \eqref{1.2} was first confirmed
in \cite{cph} with no estimates for the minimal possible degree $k_{\rm min}$ of $f$. The estimate 
$k_{\rm min}\le n^2-n$ was obtained in \cite{youn}. Ruscheweyh and Jones \cite{rujo} improved this estimate to 
$k_{\rm min}\le n-1$, which is sharp for some problems. Several different approaches to constructing 
interpolants in $\cB_{n-1}$ have been presented in \cite{gl, sw, gorh}. The paper \cite{sw}
also contains interesting results concerning minimal degree solutions. The fact that all low-degree 
solutions to the problem \eqref{1.2} are necessarily Blaschke products was established in \cite{bc}.

\smallskip

If $\f_i=\f\in\T$ for $i=1,\ldots,n$, then the constant function $f(z)\equiv \f$ is the 
only $\cB_{n-1}$-solution 
to the problem; it is not hard to show that such a problem has no other {\em rational} solutions of degree less than 
$n$. Otherwise (i.e., when at least two target values are distinct), the set of all $\cB_{n-1}$-solutions 
(or, which is the same, the set of all $\mathcal{RS}_{n-1}$-solutions) to the problem \eqref{1.2} is infinite. 
The parametrization of this set is presented in Theorem \ref{T:4.5}, the main result of the paper. In Theorem 
\ref{T:4.5a} we give a slightly different parametrization formula whch is then used in Section 4 to 
characterize the problems \eqref{1.2} having no Blaschke product solutions of degree less than $n-1$.

\section{The modified interpolation problem}
\setcounter{equation}{0}

Given a finite Blaschke product $f$ and a collection ${\bf t}=\{t_1,\ldots,t_n\}$ of distinct points in 
$\overline{\D}$, let us define the associated {\em Schwarz-Pick matrix} 
$$
P^f({\bf t})=\left[p^f(t_i,t_j)\right]_{i,j=1}^n
$$
by entry-wise formulas
\begin{equation}
p^f(t_i,t_j)=\left\{\begin{array}{cll}
t_if^\prime(t_i)\overline{f(t_i)}& \mbox{if} \; i=j \; \; \mbox{and} \; \; t_i\in\T,
\vspace{2mm}\\
{\displaystyle\frac{1-f(t_i)\overline{f(t_j)}}{1-t_i\overline{t}_j}} &
\mbox{otherwise}.\end{array}\right.
\label{2.1}
\end{equation}
In case ${\bf t}\subset\T$, we will refer to $\; P^f({\bf t})$ as to the {\em boundary Schwarz-Pick 
matrix}. Observe that for a finite Blaschke product $f$, equalities
\begin{equation}
\lim_{z\to t}\frac{1-|f(z)|^2}{1-|z|^2}=tf^\prime(t)\overline{f(t)}=|f^\prime(t)|
\label{2.2}
\end{equation}
hold at every point $t\in\T$ and thus, the diagonal entries in $P^f({\bf t})$ are all nonnegative.
The next result is well-known (see e.g. \cite[Lemma 2.1]{bpams} for the proof).
\begin{lemma}
For $f\in\cB^\circ_k$ and a tuple ${\bf t}=\{t_1,\ldots,t_n\}\in \overline{\D}^n$, 
the matrix $P^f({\bf t})$ \eqref{2.1} is positive semidefinite and
${\rm rank} \, P^f({\bf t})={\rm min}\{n,k\}$.
\label{L:2.1}
\end{lemma}
Boundary interpolation by Schur-class functions and, in particular,
by finite Blaschke products becomes more transparent if, in addition to 
conditions \eqref{1.2}, one prescribes the values of $f^\prime$ at each interpolation node $t_i$. 
We denote this modified problem by {\bf MP}.

\medskip 

{\bf MP}: {\em Given $t_i, f_i\in \T$ and
$\gamma_i\ge 0$, find a finite Blaschke product $f$ such that
\begin{equation}
f(t_i)=\f_i,\quad |f^\prime(t_i)|=\gamma_i\quad\mbox{for}\quad i=1,\ldots,n \quad (t_i, \, \f_i\in\T, 
\; \gamma_i\ge 0).
\label{2.3}
\end{equation}}
The problem \eqref{2.3} is well-known in a more general
context of rational functions $f:\D\to \D$ (\cite{bgr}) and even in the more general context 
Schur-class functions (\cite{Sarasonnp, bkiwota}). The results on finite Blaschke product interpolation 
presented in Theorem \ref{T:2.2} below are easily derived from the general ones. 
\begin{theorem}
If the problem {\bf MP} has a solution, then the {\rm Pick matrix} $P_n$ of the problem defined as
\begin{equation}
P_n=\left[p_{ij}\right]_{i,j=1}^n,\quad\mbox{where}\quad
p_{ij}=\left\{\begin{array}{cll}{\displaystyle\frac{1-\f_i\overline{\f}_j}{1-t_i\overline{t}_j}} &
\mbox{if} &  i\neq j, \vspace{1mm}\\
\gamma_i & \mbox{if} & i=j,\end{array}\right.
\label{2.4}   
\end{equation}
is positive semidefinite. Moreover,
\begin{enumerate}
\item If $P_n>0$, then the problem {\bf MP} has infinitely many solutions.
\item If $P_n\ge 0$ and ${\rm rank} \, (P_n)=k<n$, then there is an $f\in\cB_{k}$ such that 
\begin{equation}
f(t_i)=\f_i,\quad |f^\prime(t_i)|\le\gamma_i\quad\mbox{for}\quad i=1,\ldots,n
\label{2.5}
\end{equation}
and no other rational function $f: \, \D\to \D$ meets conditions \eqref{2.5}.
\end{enumerate}
\label{T:2.2}
\end{theorem}
\begin{remark}
{\rm The second statement in Theorem \ref{T:2.2} suggests a short proof of the first part of Theorem 
\ref{T:1.1}. Indeed, given $f_1,\ldots, f_n$, choose positive numbers $\gamma_1,\ldots,\gamma_{n-1}$ large
enough so that the Pick matrix $P_{n-1}$ defined via formula \eqref{2.4} is positive definite and 
extend $P_{n-1}$ to $P_n=\sbm{P_{n-1} & F \\ F^* & \gamma_n}$ by letting
$$
\gamma_n:=FP_{n-1}^{-1}F^*,\quad\mbox{where}\quad F=\begin{bmatrix}p_{n,1} & \ldots &
p_{n,n-1}\end{bmatrix}, \quad p_{n,j}=\frac{1-\f_n\overline{\f}_j}{1-t_n\overline{t}_j}.
$$
By the Schur complement argument, the matrix $P_n$ constructed as above is singular. By
Theorem \ref{T:2.2}, there is a finite Blaschke product $f\in \cB^\circ_{n-1}$ satisfying conditions 
\eqref{2.5}. Clearly, this $f$ solves the problem \eqref{1.2}.}
\label{R:2.3}
\end{remark}
It is seen from \eqref{2.1}--\eqref{2.4}, that for  every solution $f$ to the problem {\bf MP},
the Schwarz-Pick matrix $P^f({\bf t})$ based on the interpolation nodes $t_1,\ldots,t_n$ is equal to 
the matrix $P_n$ constructed in \eqref{2.4} from the interpolation data set. Hence, the first statement
in Theorem \ref{T:2.2} follows from Lemma \ref{L:2.1}. The last statement in the theorem can be 
strengthened as follows: {\em if $P_n$ is singular, there is an $f\in\cB_{k}$ satisfying conditions 
\eqref{2.5}, but there is no another Schur-class function $g$ (even non-rational) such that 
$$
\lim_{r\to 1}g(rt_i)=\f_i,\quad |\lim_{r\to 1}g^\prime(rt_i)|\le \gamma_i\quad\mbox{for}\quad 
i=1,\ldots,n.
$$}
In any event, if $P_n$ is singular, there is only one candidate which may or may not be a 
solution to the problem {\bf MP}. The determinacy criterion for the problem {\bf MP} in terms of the 
Pick matrix $P_n$ is recalled in Theorem \ref{T:2.4} below.
\begin{definition}
{\rm A positive semidefinite matrix of rank $r$ is called {\em
saturated} if every its $r\times r$ principal submatrix is positive definite. A positive semidefinite matrix
is called {\em minimally positive} if none of its diagonal entries can be decreased so that the modified matrix
will be still positive semidefinite}.
\label{D:2.3}
\end{definition}
\begin{remark}
{\rm The rank equality in Lemma \ref{L:2.1} implies that for $f\in\cB^\circ_k$ and $k<n$, the 
Schwarz-Pick matrix \eqref{2.1} is saturated}. 
\label{R:2.4}
\end{remark}
\begin{theorem}[\cite{Sarasonnp}]
The problem {\bf MP} has a unique solution if and only if $P_n$ is minimally positive or equivalently,
if and only if $P_n\ge 0$ is singular and saturated. The unique solution is a finite 
Blaschke product of degree equal the rank of $P_n$.
\label{T:2.4}
\end{theorem}

\section{Parametrization of the set of low-degree solutions}
\setcounter{equation}{0}

With the first $n-1$ interpolation conditions in \eqref{1.2} we associate the matrices 
\begin{equation}
T=\begin{bmatrix}t_1 && 0\\ &\ddots & \\ 0 && t_{n-1}\end{bmatrix},\quad M=\begin{bmatrix}\f_1 \\ \vdots \\
\f_{n-1}\end{bmatrix},\quad E=\begin{bmatrix}1 \\ \vdots \\ 1\end{bmatrix}.
\label{4.1}
\end{equation}
\begin{definition}
{\rm A tuple $\bgam=\{\gamma_1,\ldots,\gamma_{n-1}\}$ of positive numbers will be called {\em admissible}
if the Pick matrix $P_{n-1}^{\bgam}$ defined as in \eqref{2.4} is positive definite:
\begin{equation}
P^{\bgam}_{n-1}=\left[p_{ij}\right]_{i,j=1}^{n-1}>0,\quad\mbox{where}\quad
p_{ij}=\left\{\begin{array}{cll}{\displaystyle\frac{1-\f_i\overline{\f}_j}{1-t_i\overline{t}_j}} &
\mbox{if} &  i\neq j, \vspace{1mm}\\
\gamma_i & \mbox{if} & i=j.\end{array}\right.
\label{4.2}
\end{equation}}
\label{D:4.1}
\end{definition}
From now on, the tuple $\bgam=\{\gamma_1,\ldots,\gamma_{n-1}\}$ will serve as a parameter.
Observe the Stein identity \begin{equation}
P^{\bgam}_{n-1}-TP^{\bgam}_{n-1}T^*=EE^*-MM^*,
\label{4.3}   
\end{equation}
which holds true for  any choice of $\bgam=\{\gamma_1,\ldots,\gamma_{n-1}\}\in{\mathbb R}^{n-1}$.
In what follows, ${\bf e}_i$ will denote the $i$-th
column in the identity matrix $I_{n-1}$.
\begin{remark}
If $P^{\bgam}_{n-1}$ is invertible, it satisfies the Stein identity
\begin{equation}
(P^{\bgam}_{n-1})^{-1}-T^*(P^{\bgam}_{n-1})^{-1}T=XX^*-YY^*,
\label{4.11}  
\end{equation}
where the columns $X=\sbm{x_1 \\ \vdots \\ x_{n-1}}$ and
$Y=\sbm{y_1 \\ \vdots \\ y_{n-1}}$ are given by
\begin{equation} 
\begin{array}{ll}
X&=(I-t_nT^*)(P^{\bgam}_{n-1})^{-1}(t_n I-T)^{-1}E,\vspace{2mm}\\
Y&=(I-t_nT^*)(P^{\bgam}_{n-1})^{-1}(t_n I-T)^{-1}M.\end{array}
\label{4.12}  
\end{equation}
Furthermore, the entries
\begin{equation}
\begin{array}{ll}
x_i&=(1-t_n\overline{t}_i)\, {\bf e}_i^*(P^{\bgam}_{n-1})^{-1}(t_n I-T)^{-1}E,\vspace{2mm}\\
y_i&=(1-t_n\overline{t}_i)\, {\bf e}_i^*(P^{\bgam}_{n-1})^{-1}(t_n I-T)^{-1}M,
\end{array}\qquad (i=1,\ldots,n-1)
\label{4.9}
\end{equation}
in the columns \eqref{4.12} are subject to equalities
\begin{equation}
|x_i|=|y_i|\neq 0 \quad\mbox{for} \quad i=1,\ldots,n-1.
\label{4.10}
\end{equation}
\label{R:4.1}
\end{remark}
\begin{proof} Making use of \eqref{4.12} and \eqref{4.3}, we verify \eqref{4.11} as follows:
\begin{align*}
XX^*-YY^*&=
(I-t_nT^*)(P^{\bgam}_{n-1})^{-1}(t_n I-T)^{-1}\left[P^{\bgam}_{n-1}-TP^{\bgam}_{n-1}T^*\right]\\
&\qquad\times(\overline{t}_n I-T^*)^{-1}
(P^{\bgam}_{n-1})^{-1}(I-\overline{t}_nT)\\
&=(I-t_nT^*)(P^{\bgam}_{n-1})^{-1}\left[
P^{\bgam}_{n-1}(I-t_nT^*)^{-1}+\right.\\
&\qquad\left. +(I-\overline{t}_nT)^{-1}\overline{t}_nTP^{\bgam}_{n-1}\right]\times
(P^{\bgam}_{n-1})^{-1}(I-\overline{t}_nT)\\
&=(P^{\bgam}_{n-1})^{-1}(I-\overline{t}_nT)+(I-t_nT^*)(P^{\bgam}_{n-1})^{-1}\overline{t}_nT\\
&=(P^{\bgam}_{n-1})^{-1}-T^*(P^{\bgam}_{n-1})^{-1}T.
\end{align*}
Comparing the corresponding diagonal entries on both
sides of \eqref{4.11} gives $0=|x_i|^2-|y_i|^2$ so that $|x_i|=|y_i|$ for $i=1,\ldots,n-1$.
To complete the proof of \eqref{4.10}, it suffices to show that $x_i$ and $y_i$ cannot be
both equal zero. To this end, we compare the $i$-th rows of both sides in \eqref{4.11} to get
\begin{equation}
x_iX^*-y_iY^*={\bf e}_i^* \, (P^{\bgam}_{n-1})^{-1}-{\bf e}_i^* \, T^*(P^{\bgam}_{n-1})^{-1}T=
{\bf e}_i^* \, (P^{\bgam}_{n-1})^{-1}(I-\overline{t}_iT).
\label{4.13} 
\end{equation}
Let us assume that $x_i=y_i=0$. Then the expression on the right side of \eqref{4.13} is
the zero row-vector. Since $(I-\overline{t}_iT)$ is the diagonal matrix with
the $i$-th diagonal entry equal zero and all other diagonal entries being non-zero, it follows
that all entries in the row-vector ${\bf e}_i^* \, (P^{\bgam}_{n-1})^{-1}$, except the $i$-th entry, are 
zeroes, so that
${\bf e}_i^* \, (P^{\bgam}_{n-1})^{-1}=\alpha {\bf e}_i^*$ for some $\alpha\in\C$. Then it follows from
\eqref{4.9} and \eqref{4.1} that
$$
x_i=\alpha (1-t_n\overline{t}_i)\, {\bf e}_i^*(t_n I-T)^{-1}E=-\alpha\overline{t}_i.
$$
Since $x_i=0$ and $t_i\neq 0$, it follows that $\alpha=0$ and hence, ${\bf e}_i^* \,
(P^{\bgam}_{n-1})^{-1}=\alpha {\bf e}_i^*=0$. The latter cannot happen since the matrix $P^{\bgam}_{n-1}$
is invertible. The obtained contradiction completes the proof of \eqref{4.10}.
\end{proof}
For each admissible tuple $\bgam=\{\gamma_1,\ldots,\gamma_{n-1}\}$, we define the $2\times 2$ matrix 
function
\begin{align}
\Theta^{\bgam}(z)&=\begin{bmatrix}\theta^{\bgam}_{11}(z) & \theta^{\bgam}_{12}(z) \\ 
\theta^{\bgam}_{21}(z) &
\theta^{\bgam}_{22}(z)\end{bmatrix}\label{4.4}\\
&=I+(z-t_n)\begin{bmatrix}E^* \\
M^*\end{bmatrix}(I-zT^*)^{-1}(P^{\bgam}_{n-1})^{-1}(t_n I-T)^{-1}
\begin{bmatrix}E & -M\end{bmatrix},\notag
\end{align}
where $T, \, M, \, E$ are defined as in \eqref{4.1}. 
Upon making use of the columns \eqref{4.12} and of the diagonal structure of $T$, we
may write the formula \eqref{4.4} for $\Theta^{\bgam}$ as 
\begin{align}
\Theta^{\bgam}(z)&=I+(z-t_n)\begin{bmatrix}E^* \\
M^*\end{bmatrix}(I-zT^*)^{-1}(I-t_nT^*)^{-1}\begin{bmatrix}X & -Y\end{bmatrix}\label{4.14}\\
&=I+\sum_{i=1}^{n-1}\frac{z-t_n}{(1-z\overline{t}_i)(1-t_n\overline{t}_i)}\cdot
\begin{bmatrix}1 \\ \overline{\f}_i\end{bmatrix}\begin{bmatrix}x_i & -y_i\end{bmatrix}\notag\\
&=I+\sum_{i=1}^{n-1}\left(\frac{t_i}{1-z\overline{t}_i}-
\frac{t_i}{1-t_n\overline{t}_i}\right)
\cdot\begin{bmatrix}1 \\ \overline{\f}_i\end{bmatrix}\begin{bmatrix}x_i & -y_i\end{bmatrix}.
\notag
\end{align}
It is seen from \eqref{4.14} that $\Theta^{\bgam}$ is rational with simple 
poles at $t_1,\ldots,t_{n-1}$. 
We next summarize some other properties of 
$\Theta^{\bgam}$ needed for our subsequent analysis. 
\begin{theorem}
Let $P^{\bgam}_{n-1}>0$, let $X$, $Y$, $\Theta^{\bgam}$ be defined as in  \eqref{4.12}, \eqref{4.4},
and let 
\begin{equation}
\Upsilon(z)=\prod_{i=1}^{n-1}(1-z\overline{t}_i)\quad \mbox{and}\quad J=\left[\begin{array}{cr}1 & 0 \\ 0 &
-1\end{array}\right].
\label{4.5}
\end{equation}
$(1)\;$ For every unimodular constant $\f_n$, the functions
\begin{equation}
p(z)=\Upsilon(z)\left(\theta_{11}^{\bgam}(z)\f_n+\theta_{12}^{\bgam}(z)\right),\quad
q(z)=\Upsilon(z)\left(\theta_{21}^{\bgam}(z)\f_n+\theta_{22}^{\bgam}(z)\right)
\label{4.5b}
\end{equation}
are polynomials of degree $n-1$ with all zeros in $\overline{\D}$ and
in $\C\backslash\D$, respectively.

\smallskip
\noindent
$(2)\;$ The following identities hold for any $z,\zeta\in\C$ ($z\overline{\zeta}\neq 1$):
\begin{align}
\frac{J-\Theta^{\bgam}(z)J\Theta^{\bgam}(\zeta)^*}{1-z\overline{\zeta}}&=\begin{bmatrix}E^* \\
M^*\end{bmatrix}(I-zT^*)^{-1}(P^{\bgam}_{n-1})^{-1}(I-\overline{\zeta}T)^{-1}\begin{bmatrix}E
& M\end{bmatrix},\label{4.13a}\\
\frac{J-\Theta^{\bgam}(\zeta)^*J\Theta^{\bgam}(z)}{1-z\overline{\zeta}}&=\left[\begin{array}{r}X^* \\
-Y^*\end{array}\right](I-\overline{\zeta}T)^{-1}P^{\bgam}_{n-1}(I-zT^*)^{-1}\begin{bmatrix}X
& -Y\end{bmatrix}.
\label{4.16}
\end{align}
$(3)\;$  $\det \, \Theta^{\bgam}(z)=1\;$ for all $\; z\in\C\backslash\{t_1,\ldots,t_n\}$.
\label{T:4.2}
\end{theorem}
The proofs of (1)--(3) can be found in \cite{bgr} for a more general framework where 
$P^{\bgam}_{n-1}$ is an invertible Hermitian matrix satisfying the Stein identity \eqref{4.3}
for some $T$, $E$ and $M$ (though, of the same dimensions as in \eqref{4.1}) and $t_n$
is an arbitrary point in $\T\backslash\sigma(T)$. In this more general setting, 
$p$ has $\pi(P^{\bgam}_{n-1})$ zeroes in $\overline{\D}$ and $q$ has $\nu(P^{\bgam}_{n-1})$
zeroes in $\C\backslash\D$ where $\pi(P^{\bgam}_{n-1})$ and $\nu(P^{\bgam}_{n-1})$ are 
respectively the number of positive and the number of negative eigenvalues of 
$\nu(P^{\bgam}_{n-1})$ counted with multiplicities. Straightforward verifications of 
\eqref{4.13a} and \eqref{4.16} rely solely on the Stein identities \eqref{4.3} and \eqref{4.11},
respectively. Another calculation based on \eqref{4.3} shows that
$$
\det \, \Theta^{\bgam}(z)=\det \left[(zI-T)(\overline{t}_n
I-T^*)(I-zT^*)^{-1}(1-\overline{t}_n T)^{-1}\right]
$$
which is equal to one due to a special form \eqref{4.1} of $T$. 
\begin{remark} Let $P\in\C^{n\times n}$ be a positive semidefinite saturated matrix with 
${\rm rank} \, P=k<n$. Let $G\in\C^{n\times n}$ be a diagonal positive semidefinite
matrix with ${\rm rank} \, G=m<n-k$. Then ${\rm rank}(P+G)=k+m$. 
\label{R:4.5}
\end{remark}
Indeed, since the matrix $\Phi P\Phi^{-1}$ is saturated for any permutation matrix 
$\Phi$, we may take $P$ and $G$ conformally decomposed as follows:
$$
P=\begin{bmatrix} P_{11} &  P_{12} \\ P_{21} & P_{22}\end{bmatrix},\quad D=\begin{bmatrix}
0 & 0 \\ 0 &\widetilde{G}\end{bmatrix},\quad P_{11}\in\C^{k\times k}.
$$
Since ${\rm rank} \, P={\rm rank} \, P_{11}=k$ (i.e., $P_{11}$ is invertible), it follows 
the Schur complement of $P_{11}$ in $P$ is equal to the zero matrix:
$P_{22}-P_{21}P_{11}^{-1}P_{12}=0$. By the formula for the rank of a block matrix, we then have 
\begin{align*}
{\rm rank}(P+G)={\rm rank} \, \sbm{P_{11} &  P_{12} \\ P_{21} & P_{22}+\widetilde{G}}
&={\rm rank} \, P_{11}+{\rm rank}(P_{22}+\widetilde{G}-P_{21}P_{11}^{-1}P_{12})\\
&=k+{\rm rank} \, \widetilde{G}=k+{\rm rank} \, G=k+m.
\end{align*}
The next theorem is the main result of this section.
\begin{theorem}
Given data $(t_i, \, \f_i)\in\T^2$ ($i=1,\ldots,n$), let $\bgam=\{\gamma_1,\ldots,\gamma_{n-1}\}$ 
be an admissible tuple and let
\begin{equation}
f_{\bgam}(z)=\frac{\theta^{\bgam}_{11}(z)\f_n+\theta^{\bgam}_{12}(z)}{\theta^{\bgam}_{21}(z)\f_n+\theta^{\bgam}_{22}(z)}
\label{4.17}
\end{equation}
where the coefficients $\theta^{\bgam}_{ij}$ in \eqref{4.17} are constructed from $\bgam$ by formula \eqref{4.9}.
Let $x_i$ and $y_i$ be the numbers defined as in \eqref{4.9}. Then
\begin{enumerate}
\item $f_{\bgam}$ is a finite Blaschke product and satisfies conditions \eqref{1.2}.
\item $f_{\bgam}$ satisfies conditions $\; |f_{\bgam}^\prime(t_i)|\le \gamma_i \; $
for $\; i=1,\ldots,n-1$. Moreover, $\; |f_{\bgam}^\prime(t_i)|=\gamma_i\; $ if and only if $\; 
x_i\f_n\neq y_i$.
\item $\deg \, f_{\bgam}=n-1-\ell, \; $ where $\; \ell=\#\{i\in\{1,\ldots,n-1\}: \; x_i\f_n=y_i\}$.
\end{enumerate}
Conversely every Blaschke product $f\in\cB_{n-1}$ subject to interpolation conditions \eqref{1.2}
admits a representation \eqref{4.17} for some admissible tuple ${\bgam}=\{\gamma_1,\ldots,\gamma_{n-1}\}$.
This representation is unique if and only if $\deg \, f=n-1$.
\label{T:4.5}
\end{theorem}
\begin{proof} Let $\bgam=\{\gamma_1,\ldots,\gamma_{n-1}\}$ be an admissible tuple and let 
$\Theta^{\bgam}$ be 
defined as in \eqref{4.4}. Since $\Theta^{\bgam}$ is rational, the 
function $f_{\bgam}$ is rational as well. Let
\begin{equation}
N(z)=\theta^{\bgam}_{11}(z)\f_n+\theta^{\bgam}_{12}(z)\quad\mbox{and}\quad
D(z)=\theta^{\bgam}_{21}(z)\f_n+\theta^{\bgam}_{22}(z)
\label{4.18}  
\end{equation}
denote the numerator and the denominator in \eqref{4.17} and let
\begin{equation}
\Psi(z)=(I-zT^*)^{-1}(X\f_n-Y)=\sum_{i=1}^{n-1}{\bf e}_i \frac{x_i\f_n-y_i}{1-z\overline{t}_i}.
\label{4.22}  
\end{equation}
Combining \eqref{4.18} and \eqref{4.14} gives
 \begin{equation}
\begin{bmatrix}N(z) \\ D(z)\end{bmatrix}=\Theta^{\bgam}(z)\begin{bmatrix}\f_n \\ 1\end{bmatrix}
=\begin{bmatrix}\f_n \\ 1\end{bmatrix}+(z-t_n)\begin{bmatrix}E^* \\ 
M^*\end{bmatrix}(I-t_nT^*)^{-1}\Psi(z).
\label{4.19}
\end{equation}
Taking the advantage of the matrix $J$ in \eqref{4.5} and of relation \eqref{4.19} we get
\begin{align}  
|D(z)|^2-|N(z)|^2&=-\begin{bmatrix}N(z)^* & D(z)^*\end{bmatrix}J\begin{bmatrix}N(z) \\ 
D(z)\end{bmatrix}\notag\\
&=-\begin{bmatrix}\overline{\f}_n & 1\end{bmatrix}\Theta^{\bgam}(z)^*J\Theta^{\bgam}(z)
\begin{bmatrix}\f_n \\ 1\end{bmatrix}\notag\\
&=\begin{bmatrix}\overline{\f}_n & 1\end{bmatrix}\left\{J-\Theta^{\bgam}(z)^*J\Theta^{\bgam}(z)\right\}
\begin{bmatrix}\f_n \\ 1\end{bmatrix},
\label{4.21}   
\end{align}
where for the third equality we used
$$
\begin{bmatrix}\overline{\f}_n & 1\end{bmatrix}J\begin{bmatrix}\f_n \\ 1\end{bmatrix}=1-|\f_n|^2=0.
$$
Substituting \eqref{4.16} into \eqref{4.21} and making use of notation \eqref{4.22}, we conclude
$$
|D(z)|^2-|N(z)|^2=(1-|z|^2)\Psi(z)^*P^{\bgam}_{n-1}\Psi(z).
$$
Combining the latter equality with \eqref{4.17} and \eqref{4.18} gives
\begin{equation}
\frac{1-|f_{\bgam}(z)|^2}{1-|z|^2}=\frac{|D(z)|^2-|N(z)|^2}{(1-|z|^2)|D(z)|^2}
=\frac{\Psi(z)^*P^{\bgam}_{n-1}\Psi(z)}{|D(z)|^2},
\label{4.23}
\end{equation} 
which implies, in particular, that $f_{\bgam}$ is inner. Since  $f_{\bgam}$ is rational, it
extends by continuity to a finite Blaschke product. One can see from \eqref{4.22} that 
\begin{equation}
\lim_{z\to t_i}(z-t_i)\cdot \Psi(z)=-{\bf e}_it_i(x_i\f_n-y_i)
\quad \mbox{for}\quad i=1,\ldots,n-1,
\label{4.14b} 
\end{equation}  
which together with \eqref{4.19} and \eqref{4.1} implies
\begin{align}
\lim_{z\to t_i}(z-t_i)\cdot \begin{bmatrix}N(z) \\ D(z)\end{bmatrix}&=
(t_n-t_i)\begin{bmatrix}E^* \\ M^*\end{bmatrix}(I-t_nT^*)^{-1}
{\bf e}_it_i(x_i\f_n-y_i)\notag\\
&=-\begin{bmatrix}1 \\ \overline{\f}_i\end{bmatrix}t_i^2(x_i\f_n-y_i)
\quad \mbox{for}\quad i=1,\ldots,n-1.
\label{4.14c}  
\end{align} 
To show that $f_{\bgam}$ satisfies conditions \eqref{1.2}, we first observe that $\Theta^{\bgam}(t_n)=I$ 
(by definition \eqref{4.4}); now the equality $f_{\bgam}(t_n)=\f_n$ is immediate from 
\eqref{4.17}.
Verification of other equalities in  \eqref{1.2} depends on whether or not $x_i\f_n=y_i$.

\smallskip
\noindent
{\bf Case 1:} Let us assume that $x_i\f_n\neq y_i$. Then we have from \eqref{4.17} and \eqref{4.14c},
\begin{equation}
f_{\bgam}(t_i)=\frac{(z-t_i)N(z)}{(z-t_i)D(z)}=\frac{t_i^2(x_i\f_n-y_i)}{\overline{\f}_i 
t_i^2(x_i\f_n-y_i)}=\frac{1}{\overline{\f}_i}=\f_i.
\label{4.14d}
\end{equation}
Under the same assumption, we conclude from \eqref{2.2}, 
\eqref{4.2} and \eqref{4.23}--\eqref{4.14c}, 
\begin{align*}
|f_{\bgam}^\prime(t_i)|=\lim_{z\to t_i}\frac{1-|f_{\bgam}(z)|^2}{1-|z|^2}&=
\lim_{z\to t_i}\frac{|z-t_i|^2\Psi(z)^*P^{\bgam}_{n-1}\Psi(z)}{|z-t_i|^2|D(z)|^2}\notag\\
&=\frac{|x_i\f_n-y_i|^2{\bf e}_i^*P^{\bgam}_{n-1}{\bf e}_i}{|x_i\f_n-y_i|^2}=\gamma_i.
\end{align*}
{\bf Case 2:} Let us assume that $x_i\f_n= y_i$. Then the functions $N$, $D$ and $\Psi$ are analytic at 
$t_i$. To compute the values of these functions at $z=t_i$, we first take the adjoints of both sides 
in \eqref{4.13}:
$$
X\overline{x}_i-Y\overline{y}_i=(I-t_iT^*)(P^{\bgam}_{n-1})^{-1}{\bf e}_i.
$$
We next divide both sides by $\overline{y}_i$ and make use of equalities
$\f_n=y_i/x_i=\overline{x}_i/\overline{y}_i$ (by the assumption of Case 2)  
to get
$$
X\f_n-Y=\frac{1}{\overline{y}_i}(I-t_iT^*)(P^{\bgam}_{n-1})^{-1}{\bf e}_i.
$$
Substituting the latter equality into \eqref{4.22} results in
$$
\Psi(z)=\frac{1}{\overline{y}_i}(I-zT^*)^{-1}(I-t_iT^*)(P^{\bgam}_{n-1})^{-1}{\bf e}_i,  
$$
which being evaluated at $t_i$, gives
\begin{equation}
\Psi(t_i)=\frac{1}{\overline{y}_i}(I-{\bf e}_i{\bf e}_i^*)(P^{\bgam}_{n-1})^{-1}{\bf e}_i=
\frac{1}{\overline{y}_i}\left((P^{\bgam}_{n-1})^{-1}{\bf e}_i -{\bf e}_i \widetilde{p}_{ii}\right),
\label{4.25}
\end{equation}
where $\widetilde{p}_{ii}$ denotes the $i$-th diagonal entry of $(P^{\bgam}_{n-1})^{-1}$.
Evaluating the formula \eqref{4.19} at $z=t_i$ gives, in view of \eqref{4.25},
$$
\begin{bmatrix}N(t_i) \\ D(t_i)\end{bmatrix}
=\begin{bmatrix}\f_n \\ 1\end{bmatrix}+\frac{t_i-t_n}{\overline{y}_i}\begin{bmatrix}E^* \\
M^*\end{bmatrix}(I-t_nT^*)^{-1}\left((P^{\bgam}_{n-1})^{-1}{\bf e}_i -{\bf e}_i 
\widetilde{p}_{ii}\right).
$$
Making use of formulas \eqref{4.12}, we have
\begin{align*}
\frac{t_i-t_n}{\overline{y}_i}\begin{bmatrix}E^* \\ M^*\end{bmatrix}
(I-t_nT^*)^{-1}(P^{\bgam}_{n-1})^{-1}{\bf e}_i&=
\frac{t_i-t_n}{\overline{y}_i}\begin{bmatrix}X^* \\ Y^*\end{bmatrix}
(t_n I-T)^{-1}{\bf e}_i\\
&=-\frac{1}{\overline{y}_i}\begin{bmatrix}X^*{\bf e}_i \\ Y^*{\bf e}_i\end{bmatrix}=
-\frac{1}{\overline{y}_i}\begin{bmatrix}\overline{x}_i \\ \overline{y}_i\end{bmatrix}=
-\begin{bmatrix}\f_n \\ 1\end{bmatrix}.
\end{align*}
Combining the two latter formulas and again making use of \eqref{4.1} leads us to 
\begin{align*}
\begin{bmatrix}N(t_i) \\ D(t_i)\end{bmatrix}&=\frac{t_n-t_i}{\overline{y}_i}\begin{bmatrix}E^* \\ 
M^*\end{bmatrix}(I-t_nT^*)^{-1}{\bf 
e}_i\widetilde{p}_{ii}\\
&=\frac{t_n-t_i}{\overline{y}_i(1-t_n\overline{t}_i)}\begin{bmatrix}E^*{\bf e}_i \\
M^*{\bf e}_i\end{bmatrix}\widetilde{p}_{ii}=-\begin{bmatrix}1 \\ 
\overline{\f}_i\end{bmatrix}\frac{t_i\widetilde{p}_{ii}}{\overline{y}_i}.
\end{align*}
Thus, 
\begin{equation}
N(t_i)=-\frac{t_i \widetilde{p}_{ii}}{\overline{y}_i}\quad\mbox{and}\quad
D(t_i)=-\frac{t_i \widetilde{p}_{ii}\overline{\f}_i}{\overline{y}_i},
\label{4.27}
\end{equation}
and subsequently, $f_{\bgam}(t_i)=\frac{N(t_i)}{D(t_i)}=\frac{1}{\overline{\f}_i}=w_i$. Furthermore,
we have from \eqref{2.2} and \eqref{4.23}, 
\begin{equation}
|f_{\bgam}^\prime(t_i)|=\lim_{z\to t_i}\frac{1-|f_{\bgam}(z)|^2}{1-|z|^2}
=\lim_{z\to t_i}\frac{\Psi(z)^*P^{\bgam}_{n-1}\Psi(z)}{|D(z)|^2}
=\frac{\Psi(t_i)^*P^{\bgam}_{n-1}\Psi(t_i)}{|D(t_i)|^2}.
\label{4.27a}
\end{equation}
In view of \eqref{4.25} and \eqref{4.27}, 
\begin{align*}
\Psi(t_i)^*P^{\bgam}_{n-1}\Psi(t_i)&=\frac{1}{|y_i|^2}\big({\bf e}_i^*(P^{\bgam}_{n-1})^{-1}
-\widetilde{p}_{ii}{\bf e}^*_i\big)P^{\bgam}_{n-1}
\big((P^{\bgam}_{n-1})^{-1}{\bf e}_i -{\bf e}_i\widetilde{p}_{ii}\big)\\
&=\frac{1}{|y_i|^2}\big({\bf e}_i^*(P^{\bgam}_{n-1})^{-1}{\bf e}_i-2\widetilde{p}_{ii}+\widetilde{p}_{ii}^2
{\bf e}^*_iP^{\bgam}_{n-1}{\bf e}_i\big)
=\frac{\gamma_i\widetilde{p}_{ii}^2-\widetilde{p}_{ii}}{|y_i|^2},\\
|D(t_i)|^2|&=\frac{\widetilde{p}_{ii}^2}{|y_i|^2}.
\end{align*}
Substituting the two latter equalities into the right hand side of \eqref{4.27a} we get
$$
|f_{\bgam}^\prime(t_i)|=\frac{\gamma_i\widetilde{p}_{ii}^2-\widetilde{p}_{ii}}{\widetilde{p}_{ii}^2}
=\gamma_i-\frac{1}{\widetilde{p}_{ii}}<\gamma_i.
$$
We have verified equalities \eqref{1.2} and we showed that $|f_{\bgam}^\prime(t_i)|\le \gamma_i$ with 
strict inequality if and only if $x_i\f_n= y_i$ (i.e., in Case 2). This completes the proof of 
statements (1) and (2) of the theorem.

\smallskip

To prove part (3), we multiply the numerator and the denominator on the right hand side of
\eqref{4.17} by $\Upsilon$ (see formula \eqref{4.5}) to get 
a linear fractional representation for $f$ with polynomial coefficients 
$\widetilde\theta^{\bgam}_{ij}=\Upsilon \theta^{\bgam}_{ij}$:
\begin{equation}
f_{\bgam}(z)=\frac{\widetilde\theta^{\bgam}_{11}(z)\f_n+\widetilde\theta^{\bgam}_{12}(z)}
{\widetilde\theta^{\bgam}_{21}(z)\f_n+\widetilde\theta^{\bgam}_{22}(z)}=
\frac{\Upsilon(z)N(z)}{\Upsilon(z)D(z)}=\frac{p(z)}{q(z)},
\label{4.28}
\end{equation}
where $p$ and $q$ are the polynomials given in \eqref{4.5b}. Since the resulting function 
$f_{\bgam}$ extends to a finite Blaschke product (with no poles or zeroes on $\T$), it follows 
that $p$ and $q$ have the same (if any) zeroes on  $\T$ counted with multiplicities.
The common zeroes may occur only at the zeroes of the determinant of the coefficient 
matrix and since 
$\; {\displaystyle\det (\Upsilon(z)\widetilde\Theta^{\bgam}(z))=[\Upsilon(z)]^2}$
(by statement (3) in Theorem \ref{T:4.2} and by
analyticity of $\det \, (\Upsilon(z)\widetilde\Theta^{\bgam}(z))$),
it follows that $p$ and $q$ may have common zeros only at $t_1,\ldots,t_{n-1}$.
By \eqref{4.14c} and \eqref{4.5}, $p(t_i)=0$ if and only if $x_i\f_n=y_i$.
On the other hand, if this is the case, $D(t_i)\neq 0$ (by formula \eqref{4.27})
and therefore, $t_i$ is a {\em simple} zero of $p=\Upsilon D$. 
Thus, $D$ may have only simple zeros at $t_1,\ldots,t_{n-1}$.
We summarize: by statement (4) in Theorem \ref{T:4.2}, the numerator $p$ in \eqref{4.28} has $n-1$ 
zeroes in $\overline{\D}$. All zeroes of $p$ and $q$ on $\T$ are simple and common; they occur 
precisely at those $t_i$'s  for which $x_i\f_n=y_i$. After zero cancellations, the function 
$f_{\bgam}$ turns out to be a finite  Blaschke product of degree $n-1-\#\{i\in\{1,\ldots,n-1\}: \; 
x_i\f_n=y_i\}$. This completes the proof of  part (3).

\medskip

To prove the converse statement, let us assume that $f$ is a Blaschke product of degree $k\le n-1$
that satisfies conditions \eqref{1.2}. For a fixed permutation $\{i_1,\ldots,i_{n-1}\}$ of the index set 
$\{1,\ldots,n-1\}$, we choose the integers $\gamma_1,\ldots,\gamma_{n-1}$ so that
\begin{equation}
\gamma_{_{i_j}}=|f^\prime(t_{_{i_j}})|\quad (1\le j\le k)\quad\mbox{and}\quad
\gamma_{_{i_j}}>|f^\prime(t_{_{i_j}})|\quad (k< j\le n-1).
\label{4.30c}
\end{equation}
Due to this choice, the diagonal matrix
\begin{equation}
G=\begin{bmatrix}\gamma_1-|f^\prime(t_{1})| && 0 \\ &\ddots & \\ 0 && \gamma_{n-1}-|f^\prime(t_{n-1})|
\end{bmatrix}
\label{4.31}  
\end{equation}
is positive semidefinite and ${\rm rank} \, G=n-k-1$. We are going to show that the tuple 
$\bgam=\{\gamma_1,\ldots,\gamma_{n-1}\}$ is admissible and that $f=f_{\bgam}$ as in \eqref{4.17}.

\smallskip

Let $P^f=P^f(t_1,\ldots,t_{n-1},z)$ be the Schwarz-Pick matrix of $f$ based on the interpolation nodes 
$t_1,\ldots,t_{n-1}$ and one additional point $z\in\D$. According to \eqref{2.1} and due to interpolation 
conditions \eqref{1.2}, this matrix has the form
\begin{equation}
P^f=\begin{bmatrix}|f^\prime(t_1)| & \frac{1-\f_1\overline{\f}_2}{1-t_1\overline{t}_2} & \ldots &
\frac{1-\f_1\overline{\f}_{n-1}}{1-t_1\overline{t}_{n-1}} &
\frac{1-\f_1\overline{f(z)}}{1-t_1\overline{z}}\\
\frac{1-\f_2\overline{\f}_1}{1-t_2\overline{t}_1} & |f^\prime(t_2)| & \ldots &
\frac{1-\f_2\overline{\f}_{n-1}}{1-t_2\overline{t}_{n-1}}
&\frac{1-\f_2\overline{f(z)}}{1-t_2\overline{z}}\\
\vdots & \vdots & \ddots & \vdots & \vdots \\
\frac{1-\f_{n-1}\overline{\f}_{1}}{1-t_{n-1}\overline{t}_{1}}&
\frac{1-\f_{n-1}\overline{\f}_{2}}{1-t_{n-1}\overline{t}_{2}}&
\ldots & |f^\prime(t_{n-1})| & \frac{1-\f_{n-1}\overline{f(z)}}{1-t_{n-1}\overline{t}_n}\\
\frac{1-f(z)\overline{\f}_{1}}{1-z\overline{t}_{1}}&
\frac{1-f(z)\overline{\f}_{2}}{1-z\overline{t}_{2}}&\ldots &
\frac{1-f(z)\overline{\f}_{n-1}}{1-z\overline{t}_{n-1}} & \frac{1-|f(z)|^2}{1-|z|^2}\end{bmatrix}.
\label{4.29}
\end{equation}
Observe that the leading $(n-1)\times(n-1)$ submatrix of $P^f$ is the boundary Schwarz-Pick matrix
$P^f(t_1,\ldots,t_{n-1})$, while the bottom row in  $P^f$ (without the rightmost entry) 
can be written in terms of the matrices \eqref{4.1} as $\; (E^*-f(z)M^*)(I-zT^*)^{-1}$. 
Thus, $P^f$ can be written in a more compact form
\begin{equation}
P^f=\begin{bmatrix} P^f(t_1,\ldots,t_{n-1}) &  (I-\overline{z}T)^{-1}(E-M\overline{f(z)})\\
(E^*-f(z)M^*)(I-zT^*)^{-1} & {\displaystyle\frac{1-|f(z)|^2}{1-|z|^2}}\end{bmatrix}.
\label{4.30}
\end{equation}
Let $P_{n-1}^{\bgam}$ be the matrix defined via formulas \eqref{4.2} and let
\begin{equation}
\mathbb P^{\bgam}(z):=\begin{bmatrix} P_{n-1}^{\bgam} &  (I-\overline{z}T)^{-1}(E-M\overline{f(z)})\\
(E^*-f(z)M^*)(I-zT^*)^{-1} & {\displaystyle\frac{1-|f(z)|^2}{1-|z|^2}}\end{bmatrix}.
\label{4.30a}
\end{equation}
Taking into account the formula \eqref{4.31} for $G$ and comparing \eqref{4.2} and \eqref{4.30} with
\eqref{2.1} and \eqref{4.30a}, respectively, leads us to equalities
\begin{equation}
P_{n-1}^{\bgam}=P^f(t_1,\ldots,t_{n-1})+G\quad \mbox{and}\quad
\mathbb P^{\bgam}(z)=P^f+\begin{bmatrix} G & 0 \\ 0 & 0\end{bmatrix}. 
\label{4.30g}
\end{equation}
By Remark \ref{R:2.4}, the Schwarz-Pick matrices $P^f$ and 
$P^f(t_1,\ldots,t_{n-1})$ are positive semidefinite and saturated. Moreover, since 
$f\in\cB^\circ_{k}$, we have 
$$
{\rm rank} \, P^f={\rm rank} \, P^f(t_1,\ldots,t_{n-1})=k,
$$ 
by Lemma \ref{L:2.1}. Then it follows from \eqref{4.30g} by Remark \ref{R:4.5} that
\begin{equation}
{\rm rank} \, P_{n-1}^{\bgam}={\rm rank} \, P^f(t_1,\ldots,t_{n-1})+{\rm rank} \, G=n-1
\label{4.30f} 
\end{equation}
(i.e., $P_{n-1}^{\bgam}$ is positive definite and hence $\bgam=\{\gamma_1,\ldots,\gamma_{n-1}\}$ is admissible) and 
\begin{equation}
{\rm rank} \, \mathbb P^{\bgam}(z)={\rm rank} \, P^f+{\rm rank} \, \sbm{G & 0 \\ 0 & 0}=n-1
\quad\mbox{for all}\quad z\in\D.
\label{4.30b} 
\end{equation}
By \eqref{4.30f} and  \eqref{4.30b}, the Schur complement of the block $P_{n-1}^{\bgam}$ in \eqref{4.30} is equal to 
zero for every $z\in\D$:
\begin{equation}
\frac{1-|f(z)|^2}{1-|z|^2}-(E^*-f(z)M^*)(I-zT^*)^{-1}(P_{n-1}^{\bgam})^{-1}
(I-\overline{z}T)^{-1}(E-M\overline{f(z)})=0.
\label{4.41}
\end{equation} 
The rational matrix-function $\Theta^{\bgam}$ constructed from $\bgam=\{\gamma_1,\ldots,\gamma_{n-1}\}$ via formula 
\eqref{4.4}
satisfies the identity \eqref{4.13a}. Multiplying the latter identity (with $\zeta=z$)
by the row-vector $\begin{bmatrix}1 & -f(z)\end{bmatrix}$ on the left and by its adjoint on the right gives 
\begin{align*}
&(E^*-f(z)M^*)(I-zT^*)^{-1}(P_{n-1}^{\bgam})^{-1}
(I-\overline{z}T)^{-1}(E-M\overline{f(z)})\\
&=\begin{bmatrix}1 & -f(z)\end{bmatrix}\frac{J-\Theta^{\bgam}(z)J\Theta^{\bgam}(z)^*}{1-|z|^2}\begin{bmatrix}1 \\ 
-\overline{f(z)}\end{bmatrix}\\
&=\frac{1-|f(z)|^2}{1-|z|^2}-\begin{bmatrix}1 & 
-f(z)\end{bmatrix}\frac{\Theta^{\bgam}(z)J\Theta^{\bgam}(z)^*}{1-|z|^2}\begin{bmatrix}1 \\
-\overline{f(z)}\end{bmatrix}
\end{align*}
which, being combined with \eqref{4.41}, implies
\begin{equation}
\begin{bmatrix}1 & -f(z)\end{bmatrix}\frac{\Theta^{\bgam}(z)J\Theta^{\bgam}(z)^*}{1-|z|^2}
\begin{bmatrix}1 \\ -\overline{f(z)}\end{bmatrix}=0 \quad \mbox{for all} \; \; z\in\D.
\label{4.38}
\end{equation}
Let us consider the functions 
\begin{equation}
g=\theta^{\bgam}_{11}-f\theta^{\bgam}_{21}\quad\mbox{and}\quad 
\cE=\frac{f\theta^{\bgam}_{22}-\theta^{\bgam}_{12}}{\theta^{\bgam}_{11}-f\theta^{\bgam}_{21}}.
\label{4.39}
\end{equation}
The function $g$ is rational and due to \eqref{4.4}, 
$g(t_n)=\theta^{\bgam}_{11}(t_n)-f(t_n)\theta^{\bgam}_{21}(t_n)=1$. Hence, $g\not\equiv 0$ and  
the rational function $\cE$ in \eqref{4.39} is well defined. Again, due to \eqref{4.4} and the 
$n$-th interpolation condition in \eqref{1.1},
\begin{equation}
\cE(t_n)=f(t_n)\theta^{\bgam}_{22}(t_n)-\theta^{\bgam}_{12}(t_n)=f(t_n)=w_n.
\label{4.48}
\end{equation}
Using the functions \eqref{4.39} we now rewrite equality \eqref{4.38} as 
$$
0=|g(z)|^2\cdot\begin{bmatrix}1 & -\cE(z)\end{bmatrix}\frac{J}{1-|z|^2}
\begin{bmatrix}1 \\ -\overline{\cE(z)}\end{bmatrix}=\frac{|g(z)|^2(1-|\cE(z)|^2)}{1-|z|^2}.
$$
Since the latter equality holds for all $z\in\D$ and $g\not\equiv 0$, it follows that $|\cE(z)|=1$
for all $z\in\D$ so that $\cE$ is a unimodular constant. By \eqref{4.48}, $\cE\equiv w_n$. 
Now representation \eqref{4.17} follows from the the second formula in \eqref{4.39}.

\smallskip

Finally, if $k=\deg \, f<n-1$, then $n-k-1$ parameters in \eqref{4.30c} can be increased to 
produce various admissible tuples $\bgam$ such that $f=f_{\bgam}$. On the other hand, if 
$f\in\cB_{n-1}$ admits two different representations \eqref{4.17}, then for one of them 
(say, based on an admissible tuple $\bgam=\{\gamma_1,\ldots,\gamma_{n-1}\}$), we must have 
$\gamma_i\neq |f^\prime(t_i)|$ for some $i\in\{1,\ldots,n-1\}$. Then  $x_iw_n=y_i$, by part 
(2) of the theorem, and hence, $\deg \, f<n-1$, by part (3). Thus, the representation 
$f=f_{\bgam}$ is unique if and only if $\deg \, f=n-1$, which completes the proof of the 
theorem.
\end{proof}
We now reformulate Theorem \ref{T:4.5} in the form that is more convenient for numerical
computations. To this end, we let 
\begin{equation}
{\bf p}_n=\begin{bmatrix}p_{1,n}\\ \vdots \\ p_{n-1,n}\end{bmatrix},
\quad\mbox{where}\quad p_{i,n}=\frac{1-\f_i\overline{\f}_n}{1-t_i\overline{t}_n},
\label{4.47}  
\end{equation}
and, for an admissible tuple $\bgam=\{\gamma_1,\ldots,\gamma_{n-1}\}$ and the corresponding
$P_{n-1}^{\bgam}>0$, we let $\boldsymbol{\Delta}^{\bgam}=(P_{n-1}^{\bgam})^{-1}{\bf p}_n$.
If we denote by $P_{n-1,i}^{\bgam}({\bf p}_n)$ the matrix obtained from $P_{n-1}^{\bgam}$ by replacing its $i$-th column 
by ${\bf p}_n$, then by Cramer's rule, we have
\begin{equation}
\boldsymbol{\Delta^{\bgam}}=\sbm{\Delta^{\bgam}_1\\ \vdots \\ \Delta^{\bgam}_{n-1}},\qquad\Delta^{\bgam}_i=\frac{\det 
P_{n-1,i}^{\bgam}({\bf p}_n)}{\det P_{n-1}^{\bgam}}\quad
(i=1,\ldots,n-1).
\label{4.50}
\end{equation}
\begin{theorem}
Given data $(t_i, \, \f_i)\in\T^2$ ($i=1,\ldots,n$), let ${\bf p}_n$ be defined as in \eqref{4.47}.
For any admissible tuple $\bgam=\{\gamma_1,\ldots,\gamma_{n-1}\}$, the function 
\begin{equation}
f_{\bgam}(z)=\f_n\cdot \frac{1-(1-z\overline{t}_n)\cdot {\displaystyle\sum_{i=1}^{n-1}
\frac{\Delta^{\bgam}_i}{1-z\overline{t}_i}}}{1-(1-z\overline{t}_n)\cdot
{\displaystyle\sum_{i=1}^{n-1}\frac{\overline{\f}_i\f_n\Delta^{\bgam}_i}{1-z\overline{t}_i}}}
\label{4.52}
\end{equation}
with the numbers $\Delta^{\bgam}_i$ defined as in \eqref{4.50},
is the Blaschke product of degree $\deg \, f_{\bgam}=n-1-\#\{i\in\{1,\ldots,n-1\}: \; 
\Delta^{\bgam}_i=0\}$ and satisfies conditions \eqref{1.2}. Moreover, 
$|f_{\bgam}^\prime(t_i)|=\gamma_i$ if and only if $\Delta^{\bgam}_i\neq 0$ and
$|f_{\bgam}^\prime(t_i)|<\gamma_i$ otherwise.
\label{T:4.5a}
\end{theorem}
\begin{proof}
Observe that the column \eqref{4.47} can be written in terms of the matrices
\eqref{4.1} as ${\bf p}_n=(I-\overline{t}_nT)^{-1}(E-M\overline{\f}_n)$.
Let $x_i$ and $y_i$ be the numbers defined in \eqref{4.9}. From formulas \eqref{4.9} and \eqref{4.50}, we have
\begin{align}
x_i\f_n-y_i&=(1-t_n\overline{t}_i)\, {\bf e}_i^*(P^{\bgam}_{n-1})^{-1}(t_n I-T)^{-1}(E\f_n-M)
\notag\\
&=(1-t_n\overline{t}_i)\, {\bf e}_i^*(P^{\bgam}_{n-1})^{-1}(I-\overline{t}_nT)^{-1}
(E-M\overline{\f}_n)\f_n\overline{t}_n\notag \\
&=(1-t_n\overline{t}_i)\, {\bf e}_i^*(P^{\bgam}_{n-1})^{-1}{\bf p}_n\f_n\overline{t}_n\notag \\
&=(\overline{t}_n-\overline{t}_i)\, {\bf e}_i^*\boldsymbol{\Delta}^{\bgam}\f_n
=(\overline{t}_n-\overline{t}_i)\Delta_i^{\bgam} \f_n.\label{4.51}
\end{align}
Since $t_n\neq t_i$ and $w_n\neq 0$, it now follows that 
\begin{equation}
x_i\f_n=y_i \; \Longleftrightarrow \; 
\Delta^{\bgam}_i=0 \; \Longleftrightarrow \; \det P_{n-1,i}^{\bgam}({\bf p}_n)=0.
\label{4.49a}
\end{equation}
We next observe that by the second representation for $\Theta^{\bgam}$ in \eqref{4.14} and 
\eqref{4.51},
\begin{align*} 
\Theta^{\bgam}(z)\begin{bmatrix}\f_n \\ 1\end{bmatrix}
&=\begin{bmatrix}\f_n \\ 1\end{bmatrix}+
\sum_{i=1}^{n-1}\begin{bmatrix}1 \\ \overline{\f}_i\end{bmatrix}\cdot 
\frac{(z-t_n)(x_i\f_n-y_i)}{(1-z\overline{t}_i)(1-t_n\overline{t}_i)}\notag\\  
&=\begin{bmatrix}\f_n \\ 1\end{bmatrix}+
\sum_{i=1}^{n-1}\begin{bmatrix}1 \\ \overline{\f}_i\end{bmatrix}\cdot 
\frac{(z-t_n)(\overline{t}_n-\overline{t}_i)\Delta_i^{\bgam} 
\f_n}{(1-z\overline{t}_i)(1-t_n\overline{t}_i)}\notag\\
&=\begin{bmatrix}\f_n \\ 1\end{bmatrix}-\sum_{i=1}^{n-1}\begin{bmatrix}1 \\ 
\overline{\f}_i\end{bmatrix}\cdot
\frac{(1-z\overline{t}_n)\Delta^{\bgam}_i\f_n}{1-z\overline{t}_i},
\end{align*}
from which it follows that formulas \eqref{4.52} and \eqref{4.17} represent the same function $f_{\bgam}$.
Now all statements in Theorem \ref{T:4.5a} follow from their counter-parts in Theorem \ref{T:4.5}, by \eqref{4.49a}. 
\end{proof}

\section{Existence of $\cB_{n-2}$-solutions}
\setcounter{equation}{0}

We will write $(\zeta_1,\ldots,\zeta_k)\in\mathcal O$ if given $k$ points
$\zeta_1,\ldots,\zeta_k\in\T$ are counter clockwise oriented on $\T$. For example,
if $\zeta_1=1$, then $(1,\zeta_2,\ldots,\zeta_k)\in\mathcal O$ means that
$\arg \zeta_{i+1}>\arg \zeta_{i}$ for all $i=1,\ldots,k-1$. From now on, we will assume that
the interpolation nodes $t_1,\ldots, t_n$ in problem \eqref{1.2} are counter clockwise
oriented. 
\begin{theorem}
The problem \eqref{1.2} has a non-constant solution $f\in\cB_{n-2}$ if and only if there exist three target
values $\f_i,\f_j,\f_k$ having the same orientation as $t_i,t_j,t_k$ on $\T$.
\label{T:6.1}
\end{theorem}
As was pointed out in \cite{sw}, the ``only if" part follows by the winding number argument:
the absence of the requested triple means that (up to rotation of $\T$)
$\arg \f_n\le \arg \f_{n-1}\le\ldots\le \f_1$ with at least one strict 
inequality, 
and then the degree of any Blaschke product interpolation this data is at least $n-1$.
In this section we prove the ``if" part in Theorem \ref{T:6.1}.
\begin{lemma}
Given three points $\zeta_i=e^{\vartheta_i}$ ($i=1,2,3$), the quantity
$$
G(\zeta_1,\zeta_2,\zeta_3):=-i(1-\zeta_1\overline{\zeta}_2)(1-\zeta_2\overline{\zeta}_3)
(1-\zeta_3\overline{\zeta}_1)
$$
is real. Moreover,
$$
G(\zeta_1,\zeta_2,\zeta_3)>0
\Longleftrightarrow (\zeta_1,\zeta_2,\zeta_3)\in\mathcal O,\quad
G(\zeta_1,\zeta_2,\zeta_3)<0
\Longleftrightarrow (\zeta_1,\zeta_3,\zeta_2)\in\mathcal O.
$$
\label{L:5.1}
\end{lemma}
\begin{proof}  
Since $G$ is rotation-invariant, we may assume without loss of generality that
$\zeta_1=1$ (i.e., $\vartheta_1=0$). Then a straightforward computation shows that
$$
G(1,\zeta_2,\zeta_3)=8\sin{\frac{\vartheta_3-\vartheta_2}{2}}\sin{\frac{\vartheta_2}{2}}
\sin{\frac{\vartheta_3}{2}},
$$
which implies all the desired statements.
\end{proof}
\begin{corollary}
The product of any three off-diagonal entries $p_{ij}, p_{jk}, p_{ki}$ in the matrix \eqref{2.4},
$$
p_{ij}p_{jk}p_{ki}=\frac{1-\f_i\overline{\f}_j}{1-t_i\overline{t}_j}\cdot
\frac{1-\f_j\overline{\f}_k}{1-t_j\overline{t}_k}\cdot
\frac{1-\f_k\overline{\f}_i}{1-t_k\overline{t}_i}=\frac{G(\f_i,\f_j,\f_k)}{G(t_i,t_j,t_k)}
$$
is positive if and only if $\f_i,\f_j,\f_k$ are all distinct and have the same orientation on $\T$
as $t_i,t_j,t_k$.
\label{C:5.2}
\end{corollary}
\begin{proof}[Proof of Theorem \ref{T:6.1}] Let us assume that there are three target values having 
the same orientation 
on $\T$ as their respective interpolation nodes. By re-enumerating, we may assume without loss of generality that 
these values are $\f_{n-2}, \f_{n-1}, \f_n$ so that 
\begin{equation}
q:=\frac{p_{n-1,n-2}p_{n-2,n}}{p_{n-1,n}}=\frac{p_{n-1,n-2}p_{n-2,n}p_{n,n-1}}{|p_{n-1,n}|^2}>0,
\label{6.1}
\end{equation}
by Corollary \ref{C:5.2} and since $p_{n-1,n}=\overline{p}_{n,n-1}$.
We will show that in this case, there is an admissible tuple $\bgam$ so that 
the number $\Delta^{\bgam}_{n-1}$ defined in \eqref{4.50} equals zero. To this end, let 
$$
{\bf b}=\begin{bmatrix}p_{1,n-2} \\ \vdots \\ p_{n-3,n-2}\end{bmatrix},\quad
{\bf c}=\begin{bmatrix}p_{1,n-1} \\ \vdots \\ p_{n-3,n-1}\end{bmatrix},\quad 
{\bf d}=\begin{bmatrix}p_{1,n} \\ \vdots \\ p_{n-3,n}\end{bmatrix}.
$$
Let $\bgam=\{\gamma_1,\ldots,\gamma_{n-1}\}$ be any admissible tuple. Then the matrix 
$P^{\bgam}_{n-1}$ \eqref{4.2} and the column ${\bf p}_n$ \eqref{4.47} can be written as 
\begin{equation}
P^{\bgam}_{n-1}=\begin{bmatrix}P^{\bgam}_{n-3} & {\bf b} & {\bf c} \\ {\bf b}^* & \gamma_{n-2}& 
p_{n-2,n-1}\\ {\bf c}^* & p_{n-1,n-2} & \gamma_{n-1}\end{bmatrix}\quad\mbox{and}\quad 
{\bf p}_n=\begin{bmatrix}{\bf d} \\ p_{n-2,n} \\ p_{n-1,n}\end{bmatrix}.
\label{6.2}
\end{equation}
Replacing the rightmost column in $P^{\bgam}_{n-1}$ by ${\bf p}_n$ produces
\begin{equation}
P^{\bgam}_{n-1,n-1}({\bf p}_n)=\begin{bmatrix}P^{\bgam}_{n-3} & {\bf b} & {\bf d} \\ {\bf b}^* & 
\gamma_{n-2}& p_{n-2,n}\\ {\bf c}^* & p_{n-1,n-2} & p_{n-1,n}\end{bmatrix}.
\label{6.3}
\end{equation}
For any matrix $A$, we can make the entries of $(P^{\bgam}_{n-3}-A)^{-1}$ as small in modulus as we wish
by choosing the diagonal entries $\gamma_1,\ldots,\gamma_{n-3}$ in $P^{\bgam}_{n-3}$ big enough.
Thus, we choose $\gamma_1,\ldots,\gamma_{n-3}$ so huge that $P^{\bgam}_{n-3}>0$,
\begin{equation}
\det (P^{\bgam}_{n-3}-{\bf d}{\bf c}^*p_{n-1,n}^{-1})\neq 0,\quad {\bf b}^*
(P^{\bgam}_{n-3})^{-1}{\bf b}<\frac{q}{3},\quad |X|<\frac{q}{3},
\label{6.4}
\end{equation}
where $q>0$ is specified in \eqref{6.1} and where 
\begin{equation}
X=\left({\bf b}^*-\frac{p_{n-2,n}}{p_{n-1,n}}{\bf c}^*\right) 
\left(P^{\bgam}_{n-3}-{\bf d}p_{n-1,n}^{-1}{\bf c}^*\right)^{-1}\left({\bf b}-{\bf 
d}\frac{p_{n-1,n-2}}{p_{n-1,n}}\right).
\label{6.5}   
\end{equation}
Then we choose 
\begin{equation}
\gamma_{n-2}=q+X\quad\mbox{and}\quad \gamma_{n-1}\ge {\bf c}^*(P^{\bgam}_{n-3})^{-1}{\bf 
c}+\frac{3}{q}\cdot|p_{n-1,n-2}-{\bf c}^*(P^{\bgam}_{n-3})^{-1}{\bf b}|^2.
\label{6.6}   
\end{equation}
Then the tuple $\bgam=\{\gamma_1,\ldots,\gamma_{n-1}\}$ is admissible. Indeed, by \eqref{6.4} and 
\eqref{6.6}, the Schur complement of $P^{\bgam}_{n-3}$ in the matrix $P^{\bgam}_{n-2}=\sbm{P^{\bgam}_{n-3} & 
{\bf b} \\  {\bf b}^* & \gamma_{n-2}}$ is positive:
\begin{equation}
\gamma_{n-2}-{\bf b}^*(P^{\bgam}_{n-3})^{-1}{\bf b}=q+X-{\bf b}^*(P^{\bgam}_{n-3})^{-1}{\bf b}>
q-\frac{q}{3}-\frac{q}{3}=\frac{q}{3}>0,
\label{6.7}   
\end{equation}
and therefore, $P^{\bgam}_{n-2}$ is positive definite. We next use \eqref{6.7} and the second relation in 
\eqref{6.6} to show that the Schur complement of $P^{\bgam}_{n-2}$ in the 
matrix $P^{\bgam}_{n-1}$ is also positive:
\begin{align*}
&\gamma_{n-1}-\begin{bmatrix}{\bf c}^* & p_{n-1,n-2}\end{bmatrix}(P^{\bgam}_{n-2})^{-1}\begin{bmatrix}{\bf 
c} \\ p_{n-2,n-1}\end{bmatrix}\\
&=\gamma_{n-1}-{\bf c}^*(P^{\bgam}_{n-3})^{-1}{\bf c}-(\gamma_{n-2}-{\bf 
b}^*(P^{\bgam}_{n-3})^{-1}{\bf b})^{-1}\cdot|p_{n-1,n-2}-{\bf c}^*(P^{\bgam}_{n-3})^{-1}{\bf b}|^2\\
&>\gamma_{n-1}-{\bf c}^*(P^{\bgam}_{n-3})^{-1}{\bf c}-\frac{3}{q}\cdot|p_{n-1,n-2}-{\bf 
c}^*(P^{\bgam}_{n-3})^{-1}{\bf b}|^2>0.
\end{align*}
Therefore, $P^{\bgam}_{n-1}$ is positive definite. Finally, we have from \eqref{6.3}, \eqref{6.1} 
and \eqref{6.5},
\begin{align*}
&\det P^{\bgam}_{n-1,n-1}({\bf p}_n)\\
&=p_{n-1,n}\cdot \det \left(\begin{bmatrix}P^{\bgam}_{n-3} & {\bf b} \\
{\bf b}^* & \gamma_{n-2}\end{bmatrix}-\begin{bmatrix}{\bf d}\\ p_{n-2,n}\end{bmatrix}\begin{bmatrix}{\bf 
c}^* & p_{n-1,n-2}\end{bmatrix}p_{n-1,n}^{-1}\right)\\
&=p_{n-1,n}\cdot (\gamma_{n-2}-q-X)\cdot \det(P^{\bgam}_{n-3}-{\bf d}{\bf c}^*p_{n-1,n}^{-1})=0.
\end{align*}
where the last equality holds by the choice \eqref{6.6} of $\gamma_{n-2}$. By  formula \eqref{4.52},
$\Delta^{\bgam}_{n-1}=0$. Then formula \eqref{4.52}
will produce $f_{\bgam}\in\cB_{n-2}$ solving the problem \eqref{1.2}. This solution is not a constant
function since  $\f_{n-2}, \f_{n-1}, \f_n$ are all distinct.
\end{proof}

\section{Examples}
\setcounter{equation}{0}

In this section we illustrate Theorem \ref{T:4.5} by several particular examples where 
the parametrization formula \eqref{4.17} (or \eqref{4.52}) is particularly explicit in terms
of the interpolation data set. 
\subsection{Three-points problem} (cf. Example 3 in \cite{sw}). We want to find all
$f\in\cB_2$ satisfying conditions
\begin{equation}
f(t_i)=\f_i\quad (t_i,f_i\in\T, \; i=1,2,3).
\label{5.1}
\end{equation}
We exclude the trivial case where $f_1=f_2=f_3$.
By Theorem \ref{T:4.5a}, all solutions $f\in\cB_2$ to the problem
\eqref{5.1} are given by the formula
\begin{equation}
f^{\bgam}(z)=\f_3\cdot \frac{1-(1-z\overline{t}_3)\cdot \bigg(
{\displaystyle\frac{\Delta^{\bgam}_1}{1-z\overline{t}_1}}+
{\displaystyle\frac{\Delta^{\bgam}_2}{1-z\overline{t}_2}}\bigg)}
{1-(1-z\overline{t}_3)\cdot\bigg(
{\displaystyle\frac{\f_3\overline{\f}_1\Delta^{\bgam}_1}{1-z\overline{t}_1}}+
{\displaystyle\frac{\f_3\overline{\f}_2\Delta^{\bgam}_2}{1-z\overline{t}_2}}\bigg)},  
\label{5.2}
\end{equation}
where $\Delta^{\bgam}_1$ and $\Delta^{\bgam}_2$ are given, according to \eqref{4.50}, by
\begin{equation}
\Delta^{\bgam}_1=\frac{\gamma_{_2}p_{_{13}}-p_{_{12}}p_{_{23}}}{\det P_2^{\bgam}}\quad\mbox{and}\quad
\quad \Delta^{\bgam}_2=\frac{\gamma_{_1}p_{_{23}}-p_{_{21}}p_{_{13}}}{\det P_2^{\bgam}},
\label{5.3}
\end{equation}
where $p_{ij}=\frac{1-\f_i\overline{\f}_j}{1-t_i\overline{t}_j}$ and where
$\bgam=\{\gamma_1,\gamma_2\}$ is any admissible pair, i.e.,
$$
\gamma_1>0, \quad \gamma_2>0,\quad \gamma_1\gamma_2>|p_{12}|^2.
$$
Thus, any point $(\gamma_1,\gamma_2)$ in the first quadrant $\mathbb R_+^2$
above the graph of $y=|p_{12}|^2 x^{-1}$ gives rise via formula \eqref{5.2}
to a $\cB_2$-solution to the problem \eqref{5.1}. It follows immediately
by the winding number argument that there are no solutions of degree one
if $\f_1,\f_2,\f_3$ do not have the same orientation on $\T$ as $t_1,t_2,t_3$.
We can come to the same conclusion showing that in this case,
\begin{equation}
\Delta^{\bgam}_1\neq 0\quad\mbox{and}\quad\Delta^{\bgam}_2\neq 0
\label{5.4}
\end{equation}
for any admissible $\{\gamma_1,\gamma_2\}$. Indeed, if  $p_{13}=0$ (i.e., $\f_1=\f_3$), then
\eqref{5.4} holds since $p_{12}$, $p_{23}$ and $\gamma_1$ are all non-zero. Similarly,
\eqref{5.4} holds if $p_{23}=0$. If $p_{13}\neq 0$ and $p_{23}\neq 0$, then the numbers   
\begin{equation}
\widetilde{\gamma}_2:=\frac{p_{_{12}}p_{_{23}}}{p_{_{13}}}
=\frac{p_{_{12}}p_{_{23}}p_{_{31}}}{|p_{_{13}}|^2}\quad\mbox{and}\quad
\widetilde{\gamma}_1:=\frac{p_{_{21}}p_{_{13}}}{p_{_{23}}}=\frac{p_{_{21}}p_{_{13}}p_{_{32}}}{|p_{_{23}}|^2}
\label{5.5}
\end{equation}
are both non-positive, by Corollary \ref{C:5.2}. Hence, inequalities \eqref{5.4} hold for any
positive $\gamma_1,\gamma_2$ and therefore, there are no zero cancellations in \eqref{5.2}.
On the other hand, if  $\f_1,\f_2,\f_3$ have the same orientation on $\T$ as $t_1,t_2,t_3$, then
the numbers $\widetilde{\gamma}_2$ and $\widetilde{\gamma}_1$ in \eqref{5.5} are positive
(again, by Corollary \ref{C:5.2}). Observe that $\widetilde{\gamma}_1\widetilde{\gamma}_2=|p_{12}|^2$.
Therefore, any pair $(\gamma_1,\widetilde{\gamma}_2)$ with $\gamma_1>\widetilde{\gamma}_1$ is
admissible, and since
$$
\Delta^{\{\gamma_1,\widetilde{\gamma}_2\}}_1=0, \quad
\Delta^{\{\gamma_1,\widetilde{\gamma}_2\}}_2=\frac{\gamma_{_1}p_{_{23}}-p_{_{21}}p_{_{13}}}{
\gamma_{_1}\frac{p_{_{12}}p_{_{23}}}{p_{_{13}}}-|p_{_{12}}|^2}=\frac{p_{_{13}}}{p_{_{12}}},
$$
the formula \eqref{5.2} amounts to
\begin{equation}
f^{\{\gamma_1,\widetilde{\gamma}_2\}}(z)=\frac{(1-z\overline{t}_2)p_{_{12}}-(1-z\overline{t}_3)p_{_{13}}}
{(1-z\overline{t}_2)\overline{\f}_3p_{_{12}}-(1-z\overline{t}_3)\overline{\f}_2 p_{_{13}}}
\label{5.6}
\end{equation}
for any $\gamma_1>\widetilde{\gamma}_1$. On the other hand, any pair
$(\widetilde{\gamma}_1,\gamma_2)$ with $\gamma_2>\widetilde{\gamma}_2$ is admissible, and since now
$$
\Delta^{\{\widetilde{\gamma}_1,\gamma_2\}}_1=\frac{\gamma_{_2}p_{_{13}}-p_{_{12}}p_{_{23}}}{
\gamma_{_2}\frac{p_{_{21}}p_{_{13}}}{p_{_{23}}}-|p_{_{12}}|^2}=\frac{p_{_{23}}}{p_{_{21}}}
\quad\mbox{and}\quad  \Delta^{\{\widetilde{\gamma}_1,\gamma_2\}}_2=0,
$$
the formula \eqref{5.2} amounts to
\begin{equation}
f^{\{\widetilde{\gamma}_1,\gamma_2\}}(z)=\frac{(1-z\overline{t}_1)p_{_{21}}-(1-z\overline{t}_3)p_{_{23}}}
{(1-z\overline{t}_1)\overline{\f}_3p_{_{21}}-(1-z\overline{t}_3)\overline{\f}_1 p_{_{23}}}
\label{5.7}
\end{equation}
for any  $\gamma_2>\widetilde{\gamma}_2$. A straightforward verification shows that formulas
\eqref{5.6} and \eqref{5.7} define the same function (as expected, since the problem \eqref{5.1} has at
most one rational solution of degree one).
\subsection{Another example} We next consider the $n$-point problem \eqref{1.2} where all target values 
but one are equal to each other (we assume without loss of generality that this common value is $1$):
\begin{equation}
f(t_i)=1\quad (i=1,\ldots,n-1)\quad\mbox{and}\quad f(t_n)=\f_n\in\T\backslash\{1\}.
\label{3.3}
\end{equation}
\begin{proposition}
All functions $f\in\cB^\circ_{n-1}$ subject to interpolation conditions \eqref{3.3} are parametrized by the formula
\begin{equation}
f(z)=\f_n\cdot \frac{1-{\displaystyle\sum_{i=1}^{n-1}
\frac{(1-z\overline{t}_n)(1-\overline{\f}_n)}{(1-z\overline{t}_i)(1-t_i\overline{t}_n)\gamma_i}}}{1-
{\displaystyle\sum_{i=1}^{n-1}\frac{(1-z\overline{t}_n)(\f_n-1)\overline{\f}_i}
{(1-z\overline{t}_i)(1-t_i\overline{t}_n)\gamma_i}}},
\label{3.5}
\end{equation}
where positive numbers $\gamma_1,\ldots,\gamma_{n-1}$ are free parameters.
\label{P:5.1}
\end{proposition}
\begin{proof} For the data set as in \eqref{3.3}, $p_{i,j}=0$ for all $1\le i\neq j\le n-1$. Therefore, the matrix 
$P_{n-1}^{\bgam}$
is diagonal and $\bgam=\{\gamma_1,\ldots,\gamma_{n-1}\}$ is admissible if and only if $\gamma_i>0$ ($1\le i\le n-1$).
Furthermore, the numbers \eqref{4.50} are equal to
$$
\Delta^{\bgam}_i=\frac{p_{i,n}}{\gamma_i}=\frac{1-\f_i\overline{\f}_n}{\gamma_i(1-t_i\overline{t}_n)}=
\frac{1-\overline{\f}_n}{\gamma_i(1-t_i\overline{t}_n)} \quad\mbox{for}\quad i=1,\ldots,n-1.
$$
Substituting the latter formulas in \eqref{4.52} gives \eqref{3.5}. By Theorem \ref{T:4.5}, formula 
\eqref{3.5} parametrizes all $\cB_{n-1}$-solutions to the problem \eqref{3.3}. However, since $\Delta^{\bgam}_i\neq 0$
for all $i=1,\ldots,n-1$, it follows that $\deg f=n-1$  for any $f$ of the form \eqref{3.5}. 
\end{proof}
\begin{remark}
{\rm Letting $\f_n=1$ in \eqref{3.5} we see that the only $f\in\cB_{n-1}$ subject to equalities
$f(t_i)=1$ for $i=1,\ldots,n$, is the constant function $f(z)\equiv 1$.}
\label{R:6.2}
\end{remark}
We now show how to get formula \eqref{3.5} using the 
approach from \cite{bes, est} as follows.  Let $f\in\cB^\circ_{n-1}$ satisfy conditions 
\eqref{3.3} and let $\gamma_i=|f^\prime(t_i)|$ ($1\le i\le n-1$). Then $f^{-1}(\{1\})=\{t_1,\ldots,t_{n-1}\}$
and the Aleksandrov-Clark measure $\mu_{f,1}$ of $f$ at $1$ is the sum of $n$ point masses $\gamma_i^{-1}$ at $t_i$.
Therefore,
\begin{equation}
\frac{1+f(z)}{1-f(z)}=\int_\T \frac{\zeta+z}{\zeta-z} \, d\mu_{f,1}(\zeta)+ic 
=\sum_{i=1}^{n-1}\frac{1}{\gamma_i}\cdot \frac{t_i+z}{t_i-z}+ic
\label{3.2}
\end{equation}
for some $c\in\mathbb R$. Solving \eqref{3.2} for $f$ and letting $\cE=\frac{ic-1}{ic+1}$ 
(note that $\cE\in\T\backslash\{1\}$) we get
\begin{equation}
f(z)=\frac{(1-\Phi(z))\cE+\Phi(z)}{-\Phi(z)\cE+1+\Phi(z)},\quad\mbox{where}\quad
\Phi(z)=\frac{1}{2}\cdot\sum_{i=1}^{n-1}\frac{1}{\gamma_i}\cdot \frac{t_i+z}{t_i-z}.
\label{3.1}
\end{equation}
Evaluating \eqref{3.1} at $z=t_n$ and making use of the last condition in \eqref{3.3} we get
$$
\f_n=\frac{(1-\Phi(t_n))\cE+\Phi(t_n)}{-\Phi(t_n)\cE+1+\Phi(t_n)}\; \;  \Longleftrightarrow \; \; 
\cE=\frac{(1+\Phi(t_n))\f_n-\Phi(t_n)}{\Phi(t_n)\f_n+1-\Phi(t_n)}.
$$
Substituting the latter expression for $\cE$ into \eqref{3.1} leads us to the representation 
$$
f(z)=\frac{\f_n+(\Phi(z)-\Phi(t_n))(1-\f_n)}{1+(\Phi(z)-\Phi(t_n))(1-\f_n)}
$$
which is the same as \eqref{3.5}, since $|\f_n|=1$ and since according to \eqref{3.1},
$$
\Phi(z)-\Phi(t_n)=\frac{1}{2}\cdot\sum_{i=1}^{n-1}\frac{1}{\gamma_i}\cdot \left(
\frac{t_i+z}{t_i-z}-\frac{t_i+t_n}{t_i-t_n}\right)
=\sum_{i=1}^{n-1}\frac{1-z\overline{t}_n}{\gamma_i\, (1-z\overline{t}_i)(1-t_i\overline{t}_n)}.
$$
\subsection{Boundary fixed points} By Theorem \ref{T:1.1}, there are infinitely many finite Blaschke products 
$f\in\cB_{n-1}$ with given fixed boundary points $t_1,\ldots,t_n\in\T$, i.e., such that 
\begin{equation}
f(t_i)=t_i\quad \mbox{for}\quad i=1,\ldots,n.
\label{3.4}
\end{equation}
We will use Theorem \ref{T:4.5} to parametrize all such  Blaschke products.
Since $\f_i=t_i$ for $i=1,\ldots,n$, we have $p_{i,j}=1$ for all $1\le i\neq j\le n$. 
By definition \eqref{4.2}, 
\begin{equation}
P^{\bgam}_{n-1}=\Gamma+EE^*,\quad\mbox{where}\quad \Gamma=\begin{bmatrix} \gamma_1-1&&0 \\ &\ddots & \\ 0 && 
\gamma_{n-1}-1\end{bmatrix}, \; \; E=\begin{bmatrix}1 \\ \vdots \\ 1\end{bmatrix}.
\label{3.6}
\end{equation}
If the tuple $\bgam=\{\gamma_1,\ldots,\gamma_{n-1}\}$ contains two elements $\gamma_i\le 1$ and 
$\gamma_j\le 1$, then it is not admissible since the principal submatrix $\sbm {\gamma_1 & 1 \\ 1 
&\gamma_2}$ of $P^{\bgam}_{n-1}$ is not positive definite. Otherwise, that is, in one of the three 
following cases, the tuple $\bgam$ is admissible:
\begin{enumerate}
\item $\gamma_i>1$ for all $i\in\{1,\ldots,n-1\}$. 
\item $\gamma_\ell=1$ and $\gamma_i>1$ for all $i\neq \ell$.
\item $\gamma_\ell<1$, \quad $\gamma_i>1$ for all $i\neq \ell$, and $\det \, 
P^{\bgam}_{n-1}>0$.
\end{enumerate}
{\bf Cases 1\&3:} Since $\gamma_i\neq 0$ for $i=1,\ldots,n$, the matrix $\Gamma$ is invertible. By basic 
properties of determinants,
$$
\det P^{\bgam}_{n-1}=\det \Gamma\cdot\det (I+\Gamma^{-1}EE^*)=\det \Gamma\cdot (1+E^*\Gamma^{-1}E),
$$
which, on account \eqref{3.6}, implies
\begin{equation}
\det P^{\bgam}_{n-1}=\bigg(1+\sum_{i=1}^{n-1}\frac{1}{\gamma_i-1}\bigg)\cdot\prod_{i=1}^{n-1}(\gamma_i-1).
\label{3.7}   
\end{equation}
Using the latter formula, we can characterize Case 3 as follows:
\begin{equation}
\gamma_\ell<1, \quad \gamma_i>1 \quad\mbox{for all}\quad i\neq \ell, 
\quad\mbox{and}\quad \sum_{i=1}^{n-1}\frac{1}{\gamma_i-1}<-1.
\label{3.7a}
\end{equation}
Since $p_{i,n}=1$ for all $1\le i\le n-1$, we have  ${\bf p}_n=E$ in \eqref{4.47}, and hence,
$$
\det P_{n-1,j}^{\bgam}(E)=\lim_{\gamma_j\to 1}\det P^{\bgam}_{n-1}=\prod_{i\neq j}(\gamma_j-1)
\quad\mbox{for}\quad j=1,\ldots,n-1.
$$
Substituting the two latter formulas in \eqref{4.50} we get
\begin{equation}
\Delta^{\bgam}_j=\frac{\det P_{n-1,j}^{\bgam}(E)}{\det P_{n-1}^{\bgam}}
=\frac{1}{\bigg(1+{\displaystyle\sum_{i=1}^{n-1}\frac{1}{\gamma_i-1}}\bigg)\cdot (\gamma_j-1)}\qquad
(j=1,\ldots,n-1).
\label{3.8}   
\end{equation}
Substituting \eqref{3.8} into \eqref{4.52} (with $\f_i=t_i$ for $i=1,\ldots,n-1$) leads us 
(after straightforward algebraic manipulations) to the formula
\begin{equation}
f(z)=\frac{t_n+{\displaystyle\sum_{i=1}^{n-1}
\frac{z(1-t_n\overline{t}_i)}{(1-z\overline{t}_i)(\gamma_i-1)}}}{1+
{\displaystyle\sum_{i=1}^{n-1}\frac{1-t_n\overline{t}_i}   
{(1-z\overline{t}_i)(\gamma_i-1)}}}.
\label{3.9}
\end{equation}
Since $\Delta^{\bgam}_j\neq 0$ for all $j=1,\ldots,n-1$, it follows by Theorem \ref{T:4.5} that 
$f$ is a Blaschke product of degree $n-1$ and $|f^\prime(t_i)|=f^\prime(t_i)=\gamma_i$ for 
$i=1,\ldots,n-1$. Differentiating \eqref{3.9} and evaluating the 
obtained formula for $f^\prime(z)$ at $z=t_n$ gives
$$
\gamma_n:=f^\prime(t_n)=\bigg(\sum_{i=1}^{n-1}\frac{1}{\gamma_i-1}\bigg)\bigg
(1+\sum_{i=1}^{n-1}\frac{1}{\gamma_i-1}\bigg)^{-1},
$$
which can be equivalently written as
\begin{equation}
\frac{\gamma_n}{1-\gamma_n}=\sum_{i=1}^{n-1}\frac{1}{\gamma_i-1}\quad\mbox{or}\quad
\sum_{i=1}^{n}\frac{1}{\gamma_i-1}=-1.
\label{3.10}
\end{equation}  
It follows from \eqref{3.10} that in Case 1, $0<\gamma_n<1$, so that $t_n$ is the 
(hyperbolic) Denjoy-Wolff point of $f$.
Alternatively, this conclusion follows from the Cowen-Pommerenke result \cite{CP}: {\em the inequality 
\begin{equation}
\sum_{i=1}^{n-1}\frac{1}{f^{\prime}(t_i)-1}\le \frac{f^{\prime}(t_n)}{1-f^{\prime}(t_n)}
\label{3.11}
\end{equation}
for any analytic $f: \, \D\to\D$ with boundary fixed points $t_1,\ldots,t_{n-1}$ and 
the (hyperbolic) boundary fixed point $t_n$, and equality prevails in \eqref{3.11} if and only if 
$f\in\cB^\circ_{n-1}$}. Observe, that in Case 3, the point $t_\ell$ is the Denjoy-Wolff point of $f$,
while $t_n$ is a regular boundary fixed point with $\gamma_n=f^\prime(t_n)>1$.

\smallskip
\noindent
{\bf Case 2:} Direct computations show that in this case,
$$
\det P^{\bgam}_{n-1}=\prod_{j\neq \ell}(\gamma_j-1)=\det P_{n-1,\ell}^{\bgam}(E)\quad\mbox{and}
\quad \det P_{n-1,i}^{\bgam}(E)=0 \quad (i\neq \ell).
$$
Therefore, according to \eqref{4.50}, $\Delta^{\bgam}_\ell=1$ and 
$\Delta^{\bgam}_i=0 $ for all $i\neq \ell$, which being substituted into \eqref{4.52}, gives
$$
f(z)=t_n\cdot \frac{1-\frac{1-z\overline{t}_n}{1-z\overline{t}_\ell}}
{1-\frac{1-z\overline{t}_n}{1-z\overline{t}_\ell}\overline{t}_\ell t_n}=z.
$$
We summarize the preceding analysis in the next proposition.
\begin{proposition}
All functions $f\in\cB^\circ_{n-1}$ satisfying conditions \eqref{3.3} are given by 
the formula \eqref{3.9}, where the parameter $\bgam=\{\gamma_1,\ldots,\gamma_{n-1}\}$ 
is either subject to relations \eqref{3.7a} (in which case $t_\ell$ is the Denjoy-Wolff point of
$f$) or $\gamma_i>1$ for all $i$ (in which case $t_n$ is the Denjoy-Wolff point of 
$f$).
\label{P:5.8}
\end{proposition}
In particular, it follows that the identity mapping $f(z)=z$ is the only function in $\cB_{n-2}$
satisfying conditions \eqref{3.4}. In fact, a more general result in \cite{aa} asserts that 
if the $n$-point Nevanlinna-Pick problem \eqref{1.2} (with $t_i,w_i\in\C$) has a rational solution of 
degree $k<n/2$ (in the setting of \eqref{3.4}, $k=1$), then it does not have other rational solutions of degree
less than $n-k$. 

\section{Concluding remarks and open questions}
\setcounter{equation}{0}

Some partial results on the problem \eqref{1.2} can be derived from general results on rational interpolation.
Let $[x]$ denote the integer part of $x\in\mathbb R$ (the greatest integer not exceeding a given $x$).
If we let
\begin{equation}
q={\rm rank}\left[\frac{w_{r+j}-w_i}{t_{r+i}-t_j}\right]_{i,j=1}^r={\rm
rank}\left[\frac{1-w_i\overline{w}_{r+j}}{1-t_i\overline{t}_{r+j}}\right]_{i,j=1}^r,\qquad 
r=\left[\frac{n}{2}\right],
\label{6.0a}
\end{equation}
then, by a result from \cite{aa}, there are no rational functions $f$ of degree less than $q$ satisfying conditions
\eqref{1.2}. If $q\le \frac{n-1}{2}$, then there is at most one rational $f$ of degree equal $q$ satisfying 
conditions \eqref{1.2}, and the only candidate can be found in the form 
\begin{equation}
f(z)=\frac{a_0+a_1z+\ldots +a_qz^q}{b_0+b_1z+\ldots +b_qz^q}
\label{6.1a}
\end{equation}
by solving the linear system 
$$
a_0+a_1t_i+\ldots +a_qt_i^q=w_i(b_0+b_1t_i+\ldots +b_qt_i^q),\qquad i=1,\ldots,n.
$$
If there are no zero cancellations in the representation \eqref{6.1a}, $f$ satisfies all conditions in 
\eqref{1.2}. This $f$ is a finite Blaschke product if and only if it has no poles in $\overline{D}$ and 
$b_i=\overline{a}_{q-i}$ for $i=1,\ldots,n$. Alternatively, $f\in\cB_q^\circ$ if and only if the boundary 
Schwarz-Pick matrix $P^f(t_1,\ldots,t_n)$ (see \eqref{2.1}) is positive semidefinite. Another result
from from \cite{aa} states that the next possible degree of a rational solution to the problem \eqref{1.2} is 
$n-q$ and moreover, there are infinitely many rational solutions of degree $k$ for each $k\ge n-q$. 
Simple examples show that finite Blaschke product solutions of degree $n-q$ may not exist, so that 
the result from \cite{aa} provides a {\em lower bound} for the minimally possible degree of a Blashke product 
solution. On the other hand (see e.g.,  \cite{sw}), if the problem \eqref{1.2} has a solution
in $\cB^\circ_{\kappa_0}$ for $\kappa_0>\frac{n-1}{2}$, then it has infinitely solutions in $\cB^\circ_k$ 
for each $k\ge \kappa_0$. Hence, the procedure verifying whether or not the problem \eqref{1.2} has a unique
minimal degree Blaschke product solution (necessarily, $\deg f\le \frac{n-1}{2}$) is simple. The question
of some interest is to characterize the latter determinate case in terms of the original interpolation data set.  
A much more interesting question is:

\smallskip
\noindent
{\bf Question 1:} {\em Find the minimally possible $\kappa_0>\frac{n-1}{2}$ so that the problem \eqref{1.2} has 
a solution in $\cB^\circ_{\kappa_0}$. For each $k\ge \kappa$, parametrize all  $\cB^{\circ}_k$-solutions to the 
problem.} 

\medskip

In \cite{glad,glad1}, the boundary problem \eqref{1.2} was considered in the set $\mathcal{QB}_{n-1}$
of  rational functions of degree at most $n-1$ that are unimodular on $\T$. Observe that any element 
of  $\mathcal{QB}_{n-1}$ is equal to the ratio of two coprime Blashke products $g,h$ such that 
$\deg g+\deg h\le n-1$. In general, it not true that a rational function $f$ of degree less than $n$
and taking unimodular values at $n$ points on $\T$ is necessarily unimodular on $\T$ 
(by Theorem \ref{T:1.1}, this is true, if a'priori, $f$ is subject to $|f(t)|\le 1$ for all $t\in\T$).
Hence, the results concerninng low-degree unimodular interpolation do not follow directly from the general results on the 
unconstrained rational interpolation. However, we were not able to find an example providing the negative answer for the next 
question:

\smallskip
\noindent
{\bf Question 2:} {\em Let $q$ be defined as in \eqref{6.0a}, so that there are infinitely many rational
functions $f$, $\deg f=n-q$, satisfying conditions \eqref{1.2}. Is it true that some (and therefore,
infinitely many) of them are unimodular on $\T$?}

\smallskip

We finally reformulate the boundary Nevanlinna-Pick problem \eqref{1.2} in terms of a positive semidefinite 
matrix completion problem. With $t_1,\ldots,t_n$ and $\f_1,\ldots,\f_n$ in hands,
we specify all off-diagonal entries in the matrix $P_n$:
\begin{equation}
P_n=\begin{bmatrix}* & \frac{1-\f_1\overline{\f}_2}{1-t_1\overline{t}_2} & \ldots &
\frac{1-\f_1\overline{\f}_{n-1}}{1-t_1\overline{t}_{n-1}} &
\frac{1-\f_1\overline{\f}_n}{1-t_1\overline{t}_n}\\
\frac{1-\f_2\overline{\f}_1}{1-t_2\overline{t}_1} & * & \ldots &
\frac{1-\f_2\overline{\f}_{n-1}}{1-t_2\overline{t}_{n-1}}
&\frac{1-\f_2\overline{\f}_n}{1-t_2\overline{t}_n}\\
\vdots & \vdots & \ddots & \vdots & \vdots \\
\frac{1-\f_{n-1}\overline{\f}_{1}}{1-t_{n-1}\overline{t}_{1}}&
\frac{1-\f_{n-1}\overline{\f}_{2}}{1-t_{n-1}\overline{t}_{2}}&
\ldots & * & \frac{1-\f_{n-1}\overline{\f}_n}{1-t_{n-1}\overline{t}_n}\\
\frac{1-\f_{n}\overline{\f}_{1}}{1-t_{n}\overline{t}_{1}}&
\frac{1-\f_{n}\overline{\f}_{2}}{1-t_{n}\overline{t}_{2}}&\ldots &
\frac{1-\f_{n}\overline{\f}_{n-1}}{1-t_{n}\overline{t}_{n-1}} & *\end{bmatrix}.
\label{2.6}
\end{equation}
Every choice of the (ordered) set $\bgam=\{\gamma_1,\ldots,\gamma_n\}$ of real diagonal entries
produces a Hermitian completion of $P_n$ which we have denoted by  $P_n^{\bgam}$ in \eqref{4.2}.
Completion question related to the problem \eqref{1.2}  and to a similar problem in the class 
$\mathcal{QB}_{n-1}$ are the following:

\smallskip
\noindent
{\bf Question 3:} {\em Given a partially specified matrix $P_n$ \eqref{2.6}, find a positive 
semidefinite completion $P_n^{\bgam}$ with minimal possible rank}.

\medskip

The same question concerning finding the minimal rank {\em Hermitian} completion $P_n^{\bgam}$
is related to minimal degree boundary interpolation by unimodular functions. Once the minimal rank
completion is found, the finite Blaschke product $f$ with the boundary Schwarz-Pick matrix $P^f({\bf t})=P_n^{\bgam}$
will be the minimal degree solution to the problem \eqref{1.2}. Although Question 3 looks like an exercise on linear algebra, 
it turns out as difficult as the original interpolation problem.

\end{document}